\documentclass[a4paper, 11pt]{amsart}
\usepackage[utf8]{inputenc}
\usepackage[english]{babel}
\usepackage[T1]{fontenc}
\usepackage{amsmath, amssymb, amsfonts, wasysym}
\usepackage{mathtools}
\usepackage{enumitem}
\usepackage{graphicx}
\usepackage{comment}
\usepackage[hmargin=1.5cm, vmargin=2.5cm]{geometry}
\usepackage{tikz}
\usepackage[all,cmtip]{xy}
\usepackage[colorlinks=true, citecolor=blue]{hyperref}

\title{Classical actions of quantum permutation groups}
\author{Amaury Freslon}
\address{Laboratoire de Math\'ematiques d'Orsay, CNRS, Universit\'e Paris-Saclay, 91405 Orsay, France}
\email{amaury.freslon@universite-paris-saclay.fr}
\author{Frank Taipe}
\address{Instituto de Matemática y Ciencias Afines, Universidad Nacional de Ingeniería, 15012 Lima, Peru}
\email{frank.taipe@imca.edu.pe}
\author{Simeng Wang}
\address{Institute for Advanced Study in Mathematics, Harbin Institute of Technology, Harbin 150001, China}
\email{simeng.wang@hit.edu.cn}
\date{}

\subjclass[2010]{20G42, 05E10}

\theoremstyle{plain}

\newtheorem{thm}{Theorem}[section]
\newtheorem{lem}[thm]{Lemma}
\newtheorem{prop}[thm]{Proposition}
\newtheorem{cor}[thm]{Corollary}

\theoremstyle{definition}
\newtheorem{de}[thm]{Definition}
\newtheorem{ex}[thm]{Example}
\newtheorem{rem}[thm]{Remark}

\DeclareMathOperator{\Irr}{Irr}
\DeclareMathOperator{\id}{id}
\DeclareMathOperator{\Mor}{Mor}

\DeclareMathOperator{\spa}{span}
\DeclareMathOperator{\Stab}{Stab}

\newcommand{\C}{\mathbb{C}}

\newcommand{\D}{\Delta}
\newcommand{\E}{\mathbb{E}}
\newcommand{\F}{\varphi}
\newcommand{\G}{\mathbb{G}}
\newcommand{\HH}{\mathbb{H}}
\newcommand{\N}{\mathbb{N}}

\newcommand{\Z}{\mathbb{Z}}
\newcommand{\w}{\mathrm{w}}

\renewcommand{\O}{\mathcal{O}}

\usepackage{calc}
\newcounter{PartitionDepth}
\newcounter{PartitionLength}
\newsavebox{\boxpaarpart}
   \savebox{\boxpaarpart}
   { \begin{picture}(0.7,0.5)
     \put(0,0){\line(0,1){0.4}}
     \put(0.5,0){\line(0,1){0.4}}
     \put(0,0.4){\line(1,0){0.5}}
     \end{picture}}
\newsavebox{\boxbaarpart}
   \savebox{\boxbaarpart}
   { \begin{picture}(0.7,0.5)
     \put(0,0){\line(0,1){0.4}}
     \put(0.5,0){\line(0,1){0.4}}
     \put(0,0){\line(1,0){0.5}}
     \end{picture}}
\newsavebox{\boxdreipart}
   \savebox{\boxdreipart}
   { \begin{picture}(1,0.5)
     \put(0,0){\line(0,1){0.4}}
     \put(0.4,0){\line(0,1){0.4}}
     \put(0.8,0){\line(0,1){0.4}}
     \put(0,0.4){\line(1,0){0.8}}
     \end{picture}}
\newsavebox{\boxvierpart}
   \savebox{\boxvierpart}
   { \begin{picture}(1.4,0.5)
     \put(0,0){\line(0,1){0.4}}
     \put(0.4,0){\line(0,1){0.4}}
     \put(0.8,0){\line(0,1){0.4}}
     \put(1.2,0){\line(0,1){0.4}}
     \put(0,0.4){\line(1,0){1.2}}
     \end{picture}}
\newsavebox{\boxvierpartrot}
   \savebox{\boxvierpartrot}
   { \begin{picture}(0.5,0.8)
     \put(0,-0.2){\line(0,1){0.3}}
     \put(0.4,-0.2){\line(0,1){0.3}}
     \put(0,0.1){\line(1,0){0.4}}
     \put(0,0.4){\line(0,1){0.3}}
     \put(0.4,0.4){\line(0,1){0.3}}
     \put(0,0.4){\line(1,0){0.4}}
     \put(0.2,0.1){\line(0,1){0.3}}
     \end{picture}}
\newsavebox{\boxuuddpartrot}
   \savebox{\boxuuddpartrot}
   { \begin{picture}(0.5,0.8)
     \put(0,-0.2){\line(0,1){0.3}}
     \put(0.4,-0.2){\line(0,1){0.3}}
     \put(0,0.1){\line(1,0){0.4}}
     \put(0,0.4){\line(0,1){0.3}}
     \put(0.4,0.4){\line(0,1){0.3}}
     \put(0,0.4){\line(1,0){0.4}} 
     \end{picture}}
\newsavebox{\boxvierpartrotdrei}
   \savebox{\boxvierpartrotdrei}
   { \begin{picture}(1,0.8)
     \put(0,-0.2){\line(0,1){0.3}}
     \put(0.4,-0.2){\line(0,1){0.8}}
     \put(0.8,-0.2){\line(0,1){0.3}}
     \put(0,0.1){\line(1,0){0.8}}
     \end{picture}}
\newsavebox{\boxcrosspart}
   \savebox{\boxcrosspart}
   { \begin{picture}(0.5,0.8)
     \put(0,-0.2){\line(1,2){0.4}}
     \put(0.4,-0.2){\line(-1,2){0.4}}
     \end{picture}}
\newsavebox{\boxhalflibpart}
   \savebox{\boxhalflibpart}
   { \begin{picture}(1,0.8)
     \put(0,-0.2){\line(1,1){0.8}}
     \put(0.8,-0.2){\line(-1,1){0.8}}
     \put(0.4,-0.2){\line(0,1){0.8}}
     \end{picture}}
\newsavebox{\boxpositioner} 
   \savebox{\boxpositioner}
   { \begin{picture}(1.4,0.5)
     \put(0,0){\line(0,1){0.3}}
     \put(0.4,0){\line(0,1){0.6}}
     \put(0.8,0){\line(0,1){0.3}}
     \put(1.2,0){\line(0,1){0.6}}
     \put(0.4,0.6){\line(1,0){0.8}}
     \end{picture}}
\newsavebox{\boxfatcross} 
   \savebox{\boxfatcross}
   { \begin{picture}(1.4,1.3)
     \put(0,-0.2){\line(0,1){0.3}}
     \put(0.4,-0.2){\line(0,1){0.3}}
     \put(0.8,-0.2){\line(0,1){0.3}}
     \put(1.2,-0.2){\line(0,1){0.3}}
     \put(0,0.1){\line(1,0){0.4}}
     \put(0.8,0.1){\line(1,0){0.4}}
     \put(0,0.9){\line(0,1){0.3}}
     \put(0.4,0.9){\line(0,1){0.3}}
     \put(0.8,0.9){\line(0,1){0.3}}
     \put(1.2,0.9){\line(0,1){0.3}}
     \put(0,0.9){\line(1,0){0.4}}
     \put(0.8,0.9){\line(1,0){0.4}}
     \put(0.2,0.1){\line(1,1){0.8}}
     \put(0.2,0.9){\line(1,-1){0.8}}
     \end{picture}}
\newsavebox{\boxprimarypart} 
   \savebox{\boxprimarypart}
   { \begin{picture}(1,1.3)
     \put(0,-0.2){\line(0,1){0.3}}
     \put(0.4,-0.2){\line(0,1){0.3}}
     \put(0.8,-0.2){\line(0,1){0.3}}
     \put(0.4,0.1){\line(1,0){0.4}}
     \put(0,0.9){\line(0,1){0.3}}
     \put(0.4,0.9){\line(0,1){0.3}}
     \put(0.8,0.9){\line(0,1){0.3}}
     \put(0,0.9){\line(1,0){0.4}}
     \put(0,0.1){\line(1,1){0.8}}
     \put(0.2,0.9){\line(1,-2){0.4}}
     \end{picture}}
\newsavebox{\boxthreepartrot}
   \savebox{\boxthreepartrot}
   { \begin{picture}(0.5,0.5)
     \put(-0.1,-0.2){\line(0,1){0.3}}
     \put(0.3,-0.2){\line(0,1){0.3}}
     \put(-0.1,0.1){\line(1,0){0.4}}
     \put(0.1,0.1){\line(0,1){0.3}}
     \end{picture}}    
\newsavebox{\boxsixpartrot}
   \savebox{\boxsixpartrot}
   { \begin{picture}(0.5,0.8)
     \put(0,-0.2){\line(0,1){0.3}}
     \put(0.4,-0.2){\line(0,1){0.3}}
     \put(0.8,-0.2){\line(0,1){0.3}}
     \put(0,0.1){\line(1,0){0.8}}
     \put(0,0.4){\line(0,1){0.3}}
     \put(0.4,0.4){\line(0,1){0.3}}
     \put(0.8,0.4){\line(0,1){0.3}}
     \put(0,0.4){\line(1,0){0.8}}
     \put(0.4,0.1){\line(0,1){0.3}}
     \end{picture}}

\newcommand{\vierpartrot}{\usebox{\boxvierpartrot}}

\makeatother

\begin{document}

\begin{abstract}
We present a complete study of rigidity for actions of quantum permutation groups on compact spaces. In particular, we describe explicitly all actions of the quantum permutation groups on classical compact spaces, and show that the defining action is the only non-trivial ergodic one. We then extend these results to all easy quantum groups associated to non-crossing partitions.
\end{abstract}

\maketitle

\section{Introduction}
Compact quantum groups where first defined by S.L. Woronowicz in \cite{woronowicz1987compact} and \cite{woronowicz1995compact} as generalisations of compact groups. Among several aspects, it was clear that they could serve as quantum symmetries of ``non-commutative'' spaces, and this was explored immediately through the notion of quantum homogeneous space, for instance in the work of P. Podle{\'s} and S.L. Woronowicz \cite{podles1990quantum}. Besides investigating the spaces on which a given quantum group can act, it was natural to consider the converse question: given a space, which quantum group can act on it? This quantum symmetry problem was already raised  for finite spaces by Connes in his celebrated monograph \cite[End of Chapter 6]{connesNCG}, and led to the introduction of a fundamental example by S.Z. Wang in \cite{wang1998quantum}, called the \emph{quantum permutation group on $N$ points} and denoted by $S_{N}^{+}$. In particular, it is a striking fact that despite being genuinely quantum, $S_{N}^{+}$ acts on a very simple classical space, namely a space consisting of $N$ points. What is even more striking is that such a phenomenon of a classical space having quantum symmetry seems very rare. Indeed, since 2010s, Goswami discovered that any isometric action of a compact quantum group on some typical Riemannian manifold always factors through a classical group, and conjectured that the aforementioned quantum permutations are the only ``genuine'' quantum symmetries on classical compact spaces (see for instance the preprint versions of \cite{goswami2018GAFA} and also \cite{huang2013faithful}). Later on, various rigidity results were established for quantum symmetries on classical spaces with a certain differential or metric structures (see for instance \cite{goswami2020non}, \cite{chirvasitu2020existence} and \cite{chirvasitu2021generic}), but the original rigidity question for general compact topological spaces remains unsolved. Let us mention that this is not true if one allows locally compact quantum group actions, as shown in \cite{goswami2017faithful}. Therefore, the quantum rigidity phenomenon seems specific to the compact setting, which is why we stick to it.


This paper presents a complete study of the rigidity phenomena for actions on classical compact spaces by $S_N^+$ and a general class of easy quantum groups introduced in \cite{banica2009liberation}. Note that the contrast between the behavior of $S_{N}^{+}$ on discrete spaces and the apparent rigidity of connected ones led to the question whether $S_{N}^{+}$ could act non-trivially on a connected space. It was pointed out by H. Huang \cite{huang2013faithful} that this is indeed possible, and the corresponding action may be intuitively understood as quantum permutations of $N$ objects which have the same shape and glue at the fixed point space. In this paper, we prove that these are the only actions of $S_N^+$ on classical compact spaces apart from the standard action on $N$ points and the trivial action. To this end, we first completely classify all ergodic actions of $S_{N}^{+}$ on classical spaces; this is based on a refinement of the arguments in \cite{freslon2021tannaka} by the same authors, where the tools from the Tannaka--Krein reconstruction theorem for ergodic actions of C. Pinzari and J.E. Roberts \cite{pinzari2008duality} are used. That being done, we use the theory of invariant subsets for quantum group actions developped by H. Huang in \cite{huang2016invariant} to reduce the study of arbitrary actions to that of ergodic ones. This is of course not straightforward since a continuous action on a topological space is not just determined by the behaviour of its orbits, but also by their relative positions and the way they interact with the topology. However, we are able to show that the examples of Huang describe all possible classical actions of $S_{N}^{+}$.

All these results for $S_N^+$ provide a test case for the general study of classical actions of compact quantum groups. In order to provide more examples and a wider picture, we also consider the other orthogonal easy quantum groups associated to non-crossing partitions from \cite{banica2009liberation}. For most of them, the classification follows from the previous methods and results with suitable adaptation. However, for the hyperoctahedral quantum group $H_{N}^{+}$, which is known to act both on $\C^{N}$ and $\C^{2N}$, things are more involved. Note that $H_{N}^{+}$ is known to be a quantum subgroup of $S_{2N}^{+}$, hence it is not surprising that it has non-trivial classical actions, but some of them are new and do not lift to $S_{2N}^{+}$.

Let us conclude this introduction with an outline of the paper. After the preliminary section~\ref{sec:preliminaries} where we will recall some facts about compact quantum groups and their actions, we will classify in Section~\ref{sec:ergodic} all classical ergodic actions of the quantum permutation group $S_{N}^{+}$. Then, in Section~\ref{sec:general}, we will use these results to classify all actions of $S_{N}^{+}$ on classical compact spaces. Eventually, we will extend these results in Section~\ref{sec:further} to all other orthogonal easy quantum groups.

\section*{Acknowledgments}

A.F. was supported by the ANR grant ``Noncommutative analysis on groups and quantum groups'' (ANR-19-CE40-0002) and the ANR grant ``Operator algebras and dynamics on groups'' (ANR-19-CE40-0008).  S.W. was  supported by the Fundamental Research Funds for the Central Universities and the NSF of China (No.12031004, No.12301161).

\section{Preliminiaries}\label{sec:preliminaries}

In this section, we recall the necessary background on compact quantum groups and their actions which will be used in the sequel, mainly to fix notations and terminology.

\subsection{Compact quantum groups}

This work is concerned with \emph{compact quantum groups}, which were introduced by S.L. Woronowicz in \cite{woronowicz1995compact}. We refer the reader to the book \cite{neshveyev2014compact} for a detailed treatment of the theory and proofs of the results mentionned below.

\begin{de}
A \emph{compact quantum group} is a pair $(C(\G), \D)$ where $C(\G)$ is a C*-algebra and
\begin{equation*}
\D : C(\G)\to C(\G)\otimes C(\G)
\end{equation*}
is a unital $*$-homomorphism satisfying the following two conditions:
\begin{enumerate}
\item $(\D\otimes \id)\circ\D = (\id\otimes \D)\circ\D$;
\item $\overline{\spa}\{\D(C(\G))(1\otimes C(\G))\} = C(\G)\otimes C(\G) = \overline{\spa}\{\D(C(\G))(C(\G)\otimes 1)\}$.
\end{enumerate}
\end{de}

We will use the notation $\G = (C(\G), \D)$ to denote a compact quantum group. The fundamental example is given by a compact group $G$, with $C(G)$ being the C*-algebra of continuous complex-valued functions on $G$ and the coproduct being defined through the formula
\begin{equation*}
\D(f) : (g, h)\mapsto f(gh),
\end{equation*}
which makes sense through the isomorphism $C(G\times G)\cong C(G)\otimes C(G)$. This justifies the previous notation, and we will now think of $\G$ as the object underlying the ``function algebra'' $C(\G)$.

Given a compact quantum group $\G$, a \emph{representation} of $\G$ of dimension $n$ is an invertible element $u$ in the C*-algebra $M_{n}(C(\G))\simeq M_{n}(\C)\otimes C(\G)$ such that
\begin{equation*}
\D(u_{ij}) = \sum_{k=1}^{n}u_{ik}\otimes u_{kj}.
\end{equation*}
It is clear that this coincides in the case of a compact group with the usual notion of a representation. Associated to $u$ is a \emph{coaction map} $\delta_{u} : \C^{n}\to \C^{n}\otimes C(\G)$ defined by the formula
\begin{equation*}
\delta_{u}(e_{i}) = \sum_{j=1}^{n}e_{j}\otimes u_{ji},
\end{equation*}
where $(e_{i})_{1\leqslant i\leqslant n}$ is the canonical basis of $\C^{n}$. The simplest example is the \emph{trivial representation}, which is $1\in C(\G)\simeq M_{1}(C(\G))$.

Given two representations $v$ and $w$, one can form their direct sum by considering a block-diagonal matrix with blocks $v$ and $w$ respectively, and their tensor product by considering the matrix with coefficients
\begin{equation*}
(v\XBox w)_{(i,k),(j, l)} = v_{ij}w_{kl}.
\end{equation*}
In this setting, an intertwiner between two representations $v$ and $w$ of dimension respectively $n$ and $m$ will be a linear map $T : \C^{n}\to \C^{m}$ such that $Tv = wT$ (we are here identifying $M_{n}(\C)$ with $M_{n}(\C 1_{C(\G)})\subset M_{n}(C(\G))$). The set of all intertwiners between $v$ and $w$ will be denoted by $\displaystyle\Mor_{\G}(v, w)$.

If $T$ is injective, then $v$ is said to be a \emph{subrepresentation} of $w$, and if $w$ admits no subrepresentation apart from itself, then it is said to be \emph{irreducible}. One of the fundamental results in the representation theory of compact quantum groups is due to S.L. Woronowicz in \cite{woronowicz1995compact} and can be summarized as follows:

\begin{thm}[Woronowicz]
Any finite-dimensional representation of a compact quantum group splits as a direct sum of irreducible ones, and any irreducible representation is equivalent to a unitary one, i.e. the matrix of any irreducible representation of dimension $n$ is the conjugate of a unitary element of $M_{n}(C(\G))$.
\end{thm}

\subsection{Actions}

The central notion in the present work is that of an action of a compact quantum group on a C*-algebra, which we now introduce.
 
\begin{de}\label{def:action}
A {\em continuous action of a compact quantum group} $\G = (C(\G), \D)$ on a unital C*-algebra $B$ is a unital $*$-homomorphism $\alpha : B \to B \otimes C(\G)$ satisfying the following two conditions:
\begin{enumerate}
\item\label{cond:coaction} $(\alpha\otimes \id)\circ\alpha = (\id\otimes \D)\circ\alpha$;
\item\label{cond:density} $\overline{\spa}\{(1\otimes C(\G) )\alpha(B)\} = B \otimes C(\G)$.
\end{enumerate}
The corresponding {\em fixed point subalgebra} is the C*-subalgebra
\begin{equation*}
B^{\alpha} = \{b\in B \mid \alpha(b) = b\otimes 1\}
\end{equation*}
and the action $\alpha$ is called {\em ergodic} if $B^{\alpha} = \C 1$.
\end{de} 

The basic example of an action is given by the map $\delta_{u}$ introduced above, for a finite-dimensional representation $u$. Because we are only interested in continuous actions, we will drop the word ``continuous'' from now on. The reader may refer to \cite{DC17} for a comprehensive treatment of actions of compact quantum groups and proofs of the results used hereafter. One crucial feature of actions is that they can be decomposed using representations. More precisely, for a finite-dimensional representation $u$ of $\G$ of dimension $n$, let us set
\begin{equation*}
\Mor_{\G}(u, \alpha) = \left\{T : \C^{n}\to B \mid \alpha\circ T = (T\otimes\id)\circ\delta_{u}\right\}.
\end{equation*}
Then, $B_{u} = \spa\{T(\C^{n}) \mid T\in \Mor_{\G}(u, \alpha)\}$ is the largest subspace of $B$ on which the action coincides with a sum of representations equivalent to $u$. It is called the \emph{spectral subspace} associated with $u$. In particular, if $\Irr(\G)$ is a set of representatives of the equivalence classes of irreducible representations of $\G$, then the direct sum
\begin{equation*}
\mathcal{B} = \bigoplus_{\rho\in \Irr(\G)}B_{\rho}
\end{equation*}
is a dense $*$-subalgebra of $B$, which moreover satisfies $\alpha(\mathcal{B})\subset \mathcal{B}\otimes C(\G)$. The set of those $\rho$ such that $B_{\rho}\neq \{0\}$ will be called the \emph{spectrum} of the action and will be denoted by $\mathrm{Sp}(B,\alpha)$. We will frequently use the fact that $\dim B_\rho = k \dim \rho$ with some $k\leq \dim \rho$.

Assume that there is a C*-subalgebra $C(\HH)\subset C(\G)$ such that $(C(\HH), \Delta_{\mid C(\HH)})$ is a compact quantum group. Then, $\HH$ is called a \emph{quotient} of $\G$. If $\alpha$ is an action of $\G$ on a C*-algebra $B$ such that $\alpha(B)\subset B\otimes C(\HH)$, then the action is said to be \emph{induced from $\HH$}. If we want to exclude that case, we will call an action \emph{faithful} if it is not induced from any proper quotient.

\subsection{Quantum permutation groups}

We will be mainly interested in what follows by a specific example which was introduced by S.Z. Wang in \cite{wang1995free}.

\begin{de}
For $N\in \N$, we let $C(S_{N}^{+})$ denote the universal unital C*-algebra generated by $N^{2}$ elements $(p_{ij})_{1\leqslant i, j\leqslant N}$ satisfying the following relations:
\begin{itemize}
\item $p_{ij}^{2} = p_{ij} = p_{ij}^{*}$ for all $1\leqslant i, j\leqslant N$;
\item $\displaystyle\sum_{k=1}^{N}p_{ik} = 1 = \sum_{k=1}^{N}p_{kj}$ for all $1\leqslant i, j\leqslant N$.
\end{itemize}
\end{de}

It follows from the universal property that there exists a $*$-homomorphism $\D : C(S_{N}^{+})\to C(S_{N}^{+})\otimes C(S_{N}^{+})$ which is uniquely determined by the fact that
\begin{equation*}
\D(p_{ij}) = \sum_{k=1}^{N}p_{ik}\otimes p_{kj}.
\end{equation*}
As is shown in \cite{wang1998quantum}, this yields a compact quantum group called the \emph{quantum permutation group on $N$ points} and denoted by $S_{N}^{+}$. Moreover, the matrix $\rho = (p_{ij})_{1\leqslant i, j\leqslant N}$ is by construction a representation, called the {\em fundamental representation of $S_{N}^{+}$}. The complete representation theory of $S_{N}^{+}$ was computed by T. Banica in \cite{banica1999symmetries}.

\begin{thm}[Banica]\label{thm:fusionrulessn+}
For $N\geqslant 4$, the irreducible representations of $S_{N}^{+}$ can by indexed by the non-negative integers in such a way that $\rho_{0}$ is the trivial representation, $\rho = \rho_{0}\oplus \rho_{1}$ and for any $n\geqslant 1$,
\begin{equation*}
\rho_{1}\XBox\rho_{n} = \rho_{n-1}\oplus \rho_{n}\oplus\rho_{n+1}.
\end{equation*}
\end{thm}

Using the universal property, it is easy to see that there is an action of $S_{N}^{+}$ on $\C^{N}$, called {\em the standard action of $S_{N}^{+}$}, given by
\begin{equation*}
\alpha_{N}(e_{i}) = \sum_{j=1}^{N}e_{j}\otimes p_{ji},
\end{equation*}
where $(e_{i})_{1\leqslant i\leqslant N}$ is the canonical basis of $\C^{N}$. A simple calculation shows that this is an ergodic action. This can be related to the action of the classical permutation group $S_{N}$ on $N$ points. Indeed, there is a surjective $*$-homomorphism $\pi_{\text{ab}} : C(S_{N}^{+})\to C(S_{N})$ sending $p_{ij}$ to the function $\sigma\mapsto \delta_{\sigma(j)i}$. This enables to ``restrict'' the action $\alpha_{N}$ to the ``quantum subgroup'' $S_{N}$ via the map
\begin{equation*}
(\id\otimes \pi_{\text{ab}})\circ\alpha_{N} : \C^{N}\to \C^{N}\otimes C(S_{N}).
\end{equation*}
Since
\begin{align*}
((\id\otimes \pi_{\text{ab}})\circ\alpha_{N}(e_{i}))(\sigma) & = \sum_{j=1}^{N}e_{j}\otimes \delta_{\sigma(j)i} = e_{\sigma^{-1}(i)},
\end{align*}
for all $1 \leqslant i \leqslant N$, the latter map corresponds to the standard action of $S_{N}$ on $N$ points. Let us conclude by mentionning that the restriction $N\geqslant 4$ in Theorem~\ref{thm:fusionrulessn+} is necessary because for $1 \leqslant N \leqslant 3$, the map $\pi_{\text{ab}}$ is injective, hence $S_{N}^{+} = S_{N}$ in that case.

\section{Quantum permutations: the ergodic case}\label{sec:ergodic}

In this section, we will classify ergodic actions of the quantum permutation group $S_{N}^{+}$ on compact Hausdorff spaces. Let us therefore fix once and for all such a space $X$ and an action $\alpha : C(X)\to C(X)\otimes C(S_{N}^{+})$. Our strategy will be to reduce the problem to finite-dimensional actions with a specific spectrum, which we first study separately.

\subsection{Actions with small spectrum}

The purpose of this subsection is to classify ergodic actions of $S_{N}^{+}$ on finite spaces $X$ when the spectrum of the action is $\{0, 1\}$. The reason for that is the following slight extension of an argument used in the proof of \cite[Thm 6.5]{freslon2021tannaka}. Hereafter, a C*-subalgebra $B\subset C(X)$ is said to be \emph{stable under the action} if $\alpha(B)\subset B\otimes C(S_{N}^{+})$, and the corresponding action on $B$ is said to be \emph{restricted} from the original one.

\begin{lem}\label{lem:finitesubaction}
With the previous notations, there is a finite quotient $Y$ of $X$ such that $C(Y)\subset C(X)$ is stable under the action ands the restricted action has spectrum $\{0, 1\}$.
\end{lem}

\begin{proof}
Set $B = C(X)$ and let $B_{n}$ be the spectral subspace corresponding to the $n$-th irreducible representation of $S_{N}^{+}$. Assume that the action is non-trivial so that $n_{0} = \min\{n\geqslant 1 \mid B_{n}\neq \{0\}\}$ is well defined. It was shown in the proof of \cite[Thm 6.5]{freslon2021tannaka} that
\begin{equation*}
B_{n_{0}} \cdot B_{n_{0}}\subset B_{0} \oplus B_{1}.
\end{equation*}
If $n_{0} > 1$, then $B_{1} = \{0\}$ so that $A = B_{0}\oplus B_{n_{0}}$ is a finite-dimensional C*-algebra (the stability under the $*$-operation comes from the fact that the irreducible representations of $S_{N}^{+}$ are self-conjugate) such that $x\cdot y$ is scalar for all $x, y\in A$. Considering a non-trivial projection then yields a contradiction, hence we conclude that $n_{0} = 1$. But then, $B_{0}\oplus B_{1}$ is a C*-subalgebra which is stable under the action by definition, hence the result.
\end{proof}

Our purpose is to describe more precisely the actions appearing above. The only unkown parameter is the multiplicity of the representation $\rho_{1}$, which we will denote by $m$. For $m = 1$, the space acted upon has $N$ points, and this yields the standard action of $S_{N}^{+}$.

\begin{lem}\label{lem:standard}
Let $N\geqslant 4$. Then, any ergodic action of $S_{N}^{+}$ on $\C^{N}$ is isomorphic to the standard one.
\end{lem}

\begin{proof}
This is well-known, but we include a proof for the sake of completeness. By definition, there are elements $(q_{ij})_{1\leqslant i, j\leqslant N}$ in $C(S_N^+ )$ such that
\begin{equation*}
\alpha(e_{i}) = \sum_{j=1}^{N}e_{j}\otimes q_{ji}.
\end{equation*}
and it is proven in \cite[Thm 3.1, 2)]{wang1998quantum} that the elements $(q_{ij})_{1\leqslant i, j\leqslant N}$ satisfy the defining relations of $C(S_{N}^{+})$. Moreover, the equality $(\alpha\otimes\id)\circ\alpha = (\id\otimes \Delta)\circ\alpha$ implies that $\Delta(q_{ij}) = \sum q_{ik}\otimes q_{kj}$, so that the C*-subalgebra $A$ generated by these elements satisfy $\D(A)\subset A\otimes A$. By \cite{vergnioux2004k} such C*-subalgebras correspond to subrings of the fusion ring, and for $S_{N}^{+}$ an inspection of the representation theory easily shows that the only possibilities are $\C$ and $C(S_{N}^{+})$. In the first case, the action is trivial hence not ergodic. In the second case, this means that the $*$-homomorphism $\pi : C(S_{N}^{+})\to C(S_{N}^{+})$ sending $p_{ij}$ to $q_{ij}$, which exists by the universal property of $C(S_{N}^{+})$, is onto.

We claim that $\pi$, which satisfies $\Delta\circ\pi = (\pi\otimes\pi)\circ\Delta$, is an isomorphism. To see this, observe that $\rho' = (q_{ij})_{1\leqslant i, j\leqslant N}$ is an $N$-dimensional representation of $S_{N}^{+}$ which is not the sum of $N$ copies of the trivial one. For $n\geqslant 2$, the dimension of $\rho_{n}$ is greater than $(N-1)^{2} - N$ which is itself strictly greater than $N$ if the latter is at least $4$. Thus, the only possibility is that $\rho'$ is equivalent to the fundamental representation $\rho = \rho_{0}\oplus \rho_{1}$. It then follows by induction that $\pi(\rho_{n})$ is irreducible and equivalent to $\rho_{n}$. As a consequence, if $h$ denotes the Haar state of $S_{N}^{+}$, we have $h\circ\pi = h$. Because $h$ is faithful on the canonical dense Hopf $*$-algebra $\O(S_{N}^{+})$, it follows that there exists a $*$-homomorphism $\psi : \O(S_{N}^{+})\to \O(S_{N}^{+})\subset C(S_{N}^{+})$ which is inverse to $\pi_{\mid \O(S_{N}^{+})}$. By definition of the universal C*-algebra, $\psi$ extends to $C(S_{N}^{+})$ and is still an inverse for $\pi$ since it is so on a dense subalgebra. As a conclusion, $\pi$ is $*$-isomorphism which is by construction equivariant if the left-hand side acts standardly, hence the result. 
\end{proof}

\begin{rem}\label{rem:Natleast4}
The proof breaks down for $N\leqslant 3$ because the representation theory of $S_{N}^{+}$ is then different, since $S_{N}^{+} = S_{N}$ in that case. One may nevertheless observe that the statement obviously holds for $S_{2}$ and also for $S_{3}$ (a transitive action on three points has a kernel of cardinality $1$ or $2$ and since $2$ is not possible it must be faithful, hence given by an automorphism of $S_{3}$ which is necessarily inner).
\end{rem}

We will now show that this is the only possibility (together with the trivial action of course, corresponding to $m = 0$). The proof will involve the action of $S_{N}$ on specific finite sets, and we therefore first give some results about these. For the sake of clarity, we will split the results concerning $S_{N}$ into two statements, depending on the value of $N$.

\begin{prop}\label{prop:finiteactionclassical}
Let $N\geqslant 8$ be an integer and let $X$ be a set of cardinality $1 + m(N-1)$ for some integer $1\leqslant m\leqslant N-1$. If $S_{N}$ acts transitively on $X$, then $m = 1$.
\end{prop}

\begin{proof}
The argument will rely on the general classification of subgroups of $S_{N}$ with small index, which is given for instance in \cite[Chap 5]{dixon1996permutation}. More precisely, write $X = \{1, \cdots, 1 + m(N-1)\}$ and let $G < S_{N}$ be the stabilizer of $1$. Then, $G$ has index
\begin{equation*}
n = 1 + m(N-1) \leqslant 1 + (N-1)^{2}.
\end{equation*}
For $N\geqslant 8$, this is strictly less than $\binom{N}{3}$, hence by \cite[Thm 5.2 B]{dixon1996permutation} there are three possibilities:
\begin{enumerate}
\item $A_{(Y)} \leqslant G\leqslant S_{\{Y\}}$ for some subset $Y\subset X$ of cardinality at most $2$ (see below for the definition of the two groups appearing here in our context);
\item $N = 2M$ and $G$ has index $\frac{1}{2}\binom{N}{M}$;
\item $N = 8$ and $G$ has index $30$.
\end{enumerate}
We will investigate each possibility separately.

\begin{enumerate}
\item We may assume without loss of generality that $Y \subset \{1, 2\}$. If $Y = \{1, 2\}$, then $A_{(Y)} = \id_{\{1, 2\}}\times A_{N-2}$ and $S_{\{Y\}} = S_{2}\times S_{N-2}$. In particular, $G$ stabilizes $\{1, 2\}$ and its restriction to $\{3, \cdots, n\}$, contains $A_{N-2}$, hence equals $A_{N-2}$ or $S_{N-2}$. This leaves us with the following possibilities : $\id_{\{1, 2\}}\times A_{N-2}$, $\id_{\{1, 2\}}\times S_{N-2}$, $S_{2}\times A_{N-2}$, $S_{2}\times S_{N-2}$ and $S_{N-2}$. If $Y = \{1\}$, then $A_{N-1} \leqslant G \leqslant S_{N-1}$ which means that $G\in \{A_{N-1}, S_{N-1}\}$. Let us check that in each case the index cannot be equal to $n$ unless $m = 1$.
\begin{itemize}
\item $S_{N-1}$: $n = N$, implying $m = 1$;
\item $A_{N-1}$: $n=2$, impossible since $n \geqslant N$;
\item $\id_{\{1, 2\}}\times A_{N-2}$: $1+ m(N-1) = 2N(N-1)$ implies $m = 2N - 1/(N-1)$, which is not an integer for $N\geqslant 3$;
\item $\id_{\{1, 2\}}\times S_{N-2}$: $1 + m(N-1) = N(N-1)$ implies $m = N - 1/(N-1)$, which is not an integer for $N\geqslant 3$;
\item $S_{2}\times A_{N-2}$: $1 + m(N-1) = N(N-1)$, see above;
\item $S_{2}\times S_{N-2}$: $1 + m(N-1) = N(N-1)/2$ implies $m = N/2 - 1/(N-1)$ which is not an integer for $N\geqslant 3$;
\item $S_{N-2}$: $1 + m(N-1) = N(N-1)$, see above.
\end{itemize}
\item We should have
\begin{align*}
\frac{4^{M}}{\sqrt{4M}} & \leqslant \frac{1}{2}\binom{2M}{M} \\
& = 1 + m(2M-1) \\
& \leqslant 1 + (2M-1)^{2} \\
& \leqslant 4M^{2},
\end{align*}
but the inequality $4^{x} \leqslant 8x^{2}\sqrt{x}$ fails as soon as $x\geqslant 4$, i.e. for $N\geqslant 8$.
\item In this exceptional case, we have $1 + m(N-1) = 30$, hence $m(N-1) = 29$. Since $29$ is prime, this forces $m = 1$, contradicting $N = 8$.
\end{enumerate}
\end{proof}

We are now left with the case $N\leqslant 7$, which can be done by hand.

\begin{lem}\label{lem:smallN}
Let $N\leqslant 7$ be an integer and let $X$ be a set of cardinality $1 + m(N-1)$ for some integer $1\leqslant m\leqslant N-1$. If $S_{N}$ acts transitively on $X$, then $m = 1$.
\end{lem}

\begin{proof}
Note that for $N = 1, 2$ there is nothing to prove. Let $G < S_{N}$ be the stabilizer of a point. It has index $n = 1 + m(N-1)$, hence that number divides the factorial of $N$. We will therefore simply check that $n$ does not divide $N!$ if $1 < m < N$.
\begin{itemize}
\item $N = 3$: the only possibility is $n = 1 + 2\times 2 = 5$, which does not divide $6$;
\item $N = 4$: the two possibilities are $n = 1 + 2\times 3 = 7$ and $n = 1 + 3\times 3 = 10$, none of which divides $24$;
\item $N = 5$: $n\in \{9, 13, 17\}$ and these do not divide $120$;
\item $N = 6$: $n\in \{11, 16, 21, 26\}$ and among these only $16$ divides $720$, we will deal with this case below;
\item $N = 7$: $n\in \{13, 19, 25, 31, 37\}$ and these do not divide $5040$.
\end{itemize}
To conclude, we must deal with the case where $N = 6$ and $n = 16$ (hence $m = 3$). Using \cite[Thm 5.2 B]{dixon1996permutation}, we have to consider three cases. In the first one, since $16 < \binom{6}{3}$, the same reasoning as in the proof of Proposition~\ref{prop:finiteactionclassical} shows that $m = 1$. In the second case, we have $\binom{6}{3}/2 = 10\neq 16$ hence this is not possible. As for the third one, no exceptional case has index $16$ so that the proof is complete.
\end{proof}

Using this, we can elucidate the structure of finite actions of $S_{N}^{+}$. Before we give the proof, let us note an important fact. The representation $\rho = \rho_{0}\oplus\rho_{1}$ becomes, when restricted to $S_{N}$, the standard representation on $\C^{N}$. Since moreover $\rho_{0}$ becomes the only copy of the trivial representation, we conclude that $\rho_{1}$ is, when restricted to $S_{N}$, the completement of the trivial representation, and this is known to be irreducible.

\begin{prop}\label{prop:finiteaction}
Let $X$ be a finite space with an ergodic action of $S_{N}^{+}$ for $N\geqslant 4$. If the spectrum of the action is contained in $\{0, 1\}$, then either $X$ is a point with the trivial action, or $X$ has $N$ points and the action is the standard one.
\end{prop}

\begin{proof}
The action is trivial if and only if its spectrum is $\{0\}$, and we will exclude that case from now on. Consider the restricted action of $S_{N}$ on $X$. Because the irreducible representations $\rho_{0}$ and $\rho_{1}$ of $S_{N}^{+}$ are still irreducible when restricted to $S_{N}$, the spectral decomposition of the restricted action of $S_{N}$ is the same as that of the action of $S_{N}^{+}$. In particular, $\rho_{0}$ has multiplicity one, hence the action is ergodic, which in that case is equivalent to being transitive. Moreover, the multiplicity of $u^{1}$ is bounded by its dimension, which is $N-1$, so that the cardinality of $X$ equals $1 + m(N-1)$, with $0\leqslant m\leqslant N-1$. If $k = 0$, then this is the trivial action and otherwise Proposition~\ref{prop:finiteactionclassical} asserts that $m = 1$. We then conclude by Lemma~\ref{lem:standard}.
\end{proof}

Before going further, let us record an extension of Proposition~\ref{prop:finiteaction} which will be useful later on.

\begin{cor}\label{cor:spectrumk}
Let $X$ be a finite space with an ergodic action of $S_{N}^{+}$ for $N\geqslant 4$. If the spectrum of the action is contained in $\{0, 1, k\}$ for some $k > 1$, then either $X$ is a point with the trivial action, or $X$ has $N$ points and the action is the standard one.
\end{cor}

\begin{proof}
It still follows from the proof of \cite[Thm 6.5]{freslon2021tannaka} that there is a decomposition
\begin{equation*}
C(X) = B_{0}\oplus B_{1}\oplus B_{k}
\end{equation*}
where $A = B_{0}\oplus B_{1} = \C^{N}$ with its standard action by Proposition~\ref{prop:finiteaction}, and $B_{k}$ is invariant under the action and satisfies $B_{k}\cdot B_{k}\subset A$. The spectrum $Y$ of $A$ consists of $N$ points and is the range of a continuous equivariant surjection $\pi : C(X)\to C(Y)$. Considering the fibres of $\pi$ yields a partition of $X$ into $N$ subsets $X_{1}, \cdots, X_{N}$ of equal cardinality $d$, and $A$ is the algebra of functions which are constant on each $X_{i}$. We now claim that this forces $B_{k} = \{0\}$.

To see why, let us first observe that there is a function $f\in B_{k}$ such that $f_{\mid X_{i}}\neq 0$ for all $1\leqslant i\leqslant N$. Indeed, we would otherwise have that any function on $X$, being the sum of an element of $B_{k}$ and an element of $A$, is constant on at least one of the components $X_{i}$, which is absurd if $d > 1$. Let $g\in B_{k}$ and $1\leqslant i\leqslant N$. Then, $f^{2}$ and $g^{2}$ are constant on $X_{i}$, hence $f$ and $g$ are constant up to a sign. Moreover, $f\cdot g$ is constant on $X_{i}$ so that $f(x)$ and $g(x)$ are either always of the same sign or always of opposite sign for $x\in X_{i}$. We therefore conclude that $g_{\mid X_{i}} = \lambda_{i} f_{\mid X_{i}}$ for some $\lambda_{i}\in \C$. As a consequence, $B_{k}$ has dimension at most $N$, since any of its elements is determined by a scalar on each $X_{i}$. But if $B_{k}\neq \{0\}$, then its dimension is at least $\dim(\rho_{2}) = N^{2} - 3N + 1> N$ for $N\geqslant 4$, hence the result.

\end{proof}

\subsection{Arbitrary ergodic actions}

We now consider an arbitrary ergodic action
\begin{equation*}
\alpha : C(X)\to C(X)\otimes C(S_{N}^{+}).
\end{equation*}
Our goal is to show that $X$ must in fact consist in either one or $N$ points. The idea is to find invariant subalgebras of $C(X)$ with spectrum $\{0, 1, k\}$ so that we can use Corollary~\ref{cor:spectrumk}. To do this, we will refine the ideas of Lemma~\ref{lem:finitesubaction} and this first requires introducing some notations.

Recall that the fusion rules of $S_{N}^{+}$ are given by the formula
\begin{equation}\label{eq:fusionrules}
\rho_{1}\XBox \rho_{k}\cong \rho_{k-1}\oplus \rho_{k}\oplus \rho_{k+1}.
\end{equation}
There are various equivalent ways to express $\rho_{n}$ and to avoid confusion, we start by carefully fixing some conventions concerning the equality above, using the notations from \cite{freslon2021tannaka}. These notations involve non-crossing partitions to which operators and representations of $S_{N}^{+}$ are associated. The reader may refer to \cite{freslon2023compact} for a detailed exposition of that theory. In the sequel, the representation $\rho_{n}$ will always be seen as acting on the space
\begin{equation*}
H_{n}\coloneqq P_{|^{\odot n}}(\mathbb{C}^{\otimes n})\subset\mathbb{C}^{\otimes n}.
\end{equation*}
According to the proofs of \cite[Theorem 4.27 and Theorem 4.18]{freslon2013representation}, Equation \eqref{eq:fusionrules} is then realized by the sum of the ranges of the following projections for $k-1\leqslant n\leqslant k+1$: 
\begin{equation*}
P_{n}^{1,k}:H_{1}\otimes H_{k}\to H_{n}; \quad P_{n}^{1,k} = \lambda_{n}P_{|^{\odot n}}\circ Q_{v_{n}}\circ P_{h_{n}}|_{H_{1}\otimes H_{k}},
\end{equation*}
where
\begin{itemize}
\item $h_{k-1}=\usebox{\boxuuddpartrot}\odot|^{\odot k-1}$;
\item $h_{k}=\vierpartrot\odot|^{\odot k-1}$;
\item $h_{k+1}=|^{\odot k+1}$;
\item $v_{n}$ is the unique partition realizing the through-block decomposition $h_{n}=v_{n}^{*}v_{n}$;
\item $\lambda_{n}$ is a constant depending only on $k$ and $n$.
\end{itemize}
Flipping the relation above yields a similar decomposition of $u^{k}\XBox u^{1}$, with corresponding maps
\begin{equation*}
P_{n}^{k,1}:H_{k}\otimes H_{1}\to H_{n}.
\end{equation*}
With this in hand we can improve the results of \cite[Sec 6]{freslon2021tannaka}.

\begin{lem}\label{lem:ergodicactionspectrumk}
With the notations above, let $k > 1$ be an integer and consider the finite-dimensional subspace
\begin{equation*}
C_{k} = B_{0}\oplus B_{1}\oplus B_{k}.
\end{equation*}
Then, $C_{k}$ is a C*-subalgebra of $C(X)$.
\end{lem}

\begin{proof}
First note that $C_{k}^{*} = C_{k}$ because all the irreducible representations of $S_{N}^{+}$ are self-conjugate. Recall moreover that it follows from \cite[Lem 6.1 and Prop 6.4]{freslon2021tannaka} that for any $k\geqslant 1$, $B_{k}\cdot B_{k}\subset B_{0}\oplus B_{1}$. Thus, the only thing that we have to prove is that $B_{1}\cdot B_{k}\subset C_{k}$. This will be done with arguments similar to those of \cite{freslon2021tannaka}.

Recall that there is an equivariant $*$-isomorphism
\begin{equation*}
B\cong\bigoplus_{n\in\N}\overline{\F(\rho_{n})}\otimes H_{n}
\end{equation*}
and that for $\xi_{1}\in\varphi(\rho_{1}),\xi_{2}\in\F(\rho_{k})$ and $\eta_{1}\in H_{1},\eta_{2}\in H_{k}$,
the product reads 
\begin{equation*}
(\overline{\xi}_{1}\otimes\eta_{1})\cdot(\overline{\xi}_{2}\otimes\eta_{2}) = \sum_{n=k-1}^{k+1}\overline{\F(P_{n}^{1,k})\circ\iota_{1,k}(\xi_{1}\otimes\xi_{2})}\otimes P_{n}^{1,k}(\eta_{1}\otimes\eta_{2}).
\end{equation*}
Note that the sum above corresponds to the direct sum in the decomposition formula of $B$, and we used here the fact that any irreducible subrepresentation $\rho_{n}\subset \rho_{1}\otimes \rho_{k}$ has multiplicity one. A similar formula holds for $(\overline{\xi}_{2}\otimes\eta_{2})\cdot(\overline{\xi}_{1}\otimes\eta_{1})$. Because $B$ is assumed to be commutative, we have for any choices of $\xi_{1},\xi_{2},\eta_{1},\eta_{2}$, 
\begin{equation*}
\sum_{n=k-1}^{k+1}\overline{\F(P_{n}^{1,k})\circ\iota_{1,k}(\xi_{1}\otimes\xi_{2})}\otimes P_{n}^{1,k}(\eta_{1}\otimes\eta_{2})=\sum_{n=k-1}^{k+1}\overline{\F(P_{n}^{k,1})\circ\iota_{k,1}(\xi_{2}\otimes\xi_{1})}\otimes P_{n}^{k,1}(\eta_{2}\otimes\eta_{1}).
\end{equation*}
As in the proof of \cite[Lem 6.1]{freslon2021tannaka}, we conclude that if the $n$-th component in the sum is non-zero, then
\begin{equation*}
P_{n}^{1,k}(\eta_{1}\otimes\eta_{2}) = \pm P_{n}^{k,1}(\eta_{2}\otimes\eta_{1})
\end{equation*}
for $k-1\leqslant n\leqslant k+1$ and for all $\eta_{1},\eta_{2}$ as soon as $\F(P_{n}^{k,1})\circ\iota\neq 0$ (which is equivalent to $\F(P_{n}^{1,k})\circ\iota\neq 0$). To conclude, we have to prove that the condition above fails for $n = k-1$ and $n = k+1$.

We will focus on the case $n=k-1$, the case $n=k+1$ being similar. We have 
\begin{equation*}
P_{k-1}^{1,k} = \lambda_{k-1}P_{|^{\odot k-1}}\circ Q_{\sqcup\odot|^{\odot k-1}}\circ P_{\usebox{\boxuuddpartrot}\odot|^{\odot k-1}}|_{H_{1}\otimes H_{k}}.
\end{equation*}
According to \cite[Lem 5.7]{freslon2013representation}, any $q\prec\usebox{\boxuuddpartrot}\odot|^{\odot k-1}$ must be of the form $\usebox{\boxuuddpartrot}\odot q'$ for some $q'\prec|^{\odot k-1}$. Thus, $P_{\usebox{\boxuuddpartrot}\odot|^{\odot k-1}}=S_{\usebox{\boxuuddpartrot}}\otimes P_{|^{\odot k-1}}$,
and consequently 
\begin{equation*}
P_{k-1}^{1,k}=\lambda_{k-1}(Q_{\sqcup}\otimes P_{|^{\odot k-1}})|_{H_{1}\otimes H_{k}}
\end{equation*}
for some nonzero constant $\lambda_{k-1}$ depending only on $k$. Symmetrically, we have 
\begin{equation*}
P_{k-1}^{k,1}=\lambda_{k-1}(P_{|^{\odot k-1}}\otimes Q_{\sqcup})|_{H_{k}\otimes H_{1}}.
\end{equation*}
As in the proof of  \cite[Prop 6.4]{freslon2021tannaka} , we let $v_1 ,v_2 , v_3$ be an orthonormal basis of $\mathrm{span} \{e_i :1\leq i\leq 4 \} \cap \xi ^\bot $ given by
\[
v_1 = \frac{1}{\sqrt 2} e_1 - \frac{1}{\sqrt 2}e_2,\quad
v_2 = \frac{1}{\sqrt 6} e_1 + \frac{1}{\sqrt 6} e_2 - \frac{2}{\sqrt 6} e_3, \quad
v_3 = \frac{1}{2\sqrt 3} e_1 + \frac{1}{2\sqrt 3} e_2 +\frac{1}{2\sqrt 3} e_3 - \frac{3}{2\sqrt 3} e_4,
\]
and
\begin{equation*}
A_{\mathbf{i}} = v_{i_{1}}\otimes \cdots\otimes v_{i_{k}}
\end{equation*}. Now we take $\eta_{1} = e_{1}-e_{2}\in H_{1}$ and $\eta_{2}=A_{\mathrm{{i}}}\in H_{k}$
with $\boldsymbol{\mathrm{i}}=(1,2,1,2,\cdots)$. Set $\boldsymbol{\mathrm{j}}=(2,1,2,1,\cdots)$
of length $k-1$. Then
\begin{align*}
\langle P_{k-1}^{1,k}(\eta_{1}\otimes\eta_{2}),A_{\boldsymbol{\mathrm{j}}}\rangle & =\lambda_{k-1}'\langle(e_{1}-e_{2})\otimes A_{\mathrm{\boldsymbol{i}}},\sum_{i}e_{i}\otimes e_{i}\otimes A_{\boldsymbol{\mathrm{j}}}\rangle\\
 & =\lambda_{k-1}'\|A_{\boldsymbol{\mathrm{j}}}\|^{2},
\end{align*}
but
\[
\langle P_{k-1}^{k,1}(\eta_{2}\otimes\eta_{1}),A_{\boldsymbol{\mathrm{j}}}\rangle=\lambda_{k-1}'\langle A_{\mathrm{\boldsymbol{i}}}\otimes(e_{1}-e_{2}),A_{\boldsymbol{\mathrm{j}}}\otimes\sum_{i}e_{i}\otimes e_{i}\rangle=0.
\]
This means that $P_{n}^{1,k}(\eta_{1}\otimes\eta_{2})\neq\pm P_{n}^{k,1}(\eta_{2}\otimes\eta_{1})$ and hence both are $0$.
\end{proof}

We are now ready to prove the first main result of this work, elucidating classical ergodic actions of $S_{N}^{+}$.

\begin{thm}\label{thm:ergodic}
Let $X$ be a compact Hausdorff space and let $\alpha : C(X)\to C(X)\otimes C(S_{N}^{+})$ be an ergodic action, for $N\geqslant 4$. Then, up to isomorphism, $\alpha$ is either the trivial action on a point or the standard action on $N$ points.
\end{thm}

\begin{proof}
For $k\geqslant 2$, we have by Lemma~\ref{lem:ergodicactionspectrumk} that $C_{k}$ is a finite-dimensional subalgebra, which is by construction invariant under the action. Moreover, by construction the spectrum of the action is contained in $\{0, 1, k\}$. By Corollary~\ref{cor:spectrumk}, we conclude that $C_{k}$ is $N$-dimensional, which means that $B_{k} = \{0\}$. Since this is true for all $k\geqslant 2$, $\mathrm{Sp}(C(X),\alpha)\subset \{0, 1\}$ and we conclude by Proposition~\ref{prop:finiteaction}.
\end{proof}

\begin{rem}
Once again, the result still holds trivially of $N = 2$. However, it fails for $N = 3$ since for instance $S_{3}^{+} = S_{3}$ acts transitively on itself.
\end{rem}

\begin{rem}
There is an alternative argument in the case of finite actions, which was indicated to us by M. Yamashita. Ergodic actions of a given compact quantum group on finite-dimensional C*-algebras are called \emph{torsion actions}, because of their role in the categorical formulation of the Baum-Connes conjecture for quantum groups (see \cite{meyer2006baum}). Torsion actions of $S_{N}^{+}$ were classified by C. Voigt in \cite{voigt2015structure}, and it turns out that there are exactly two of them (the trivial action on $\C$ and the standard action on $\C^{N}$) up to equivariant Morita equivalence. This means that given a finite space $X$ together with an ergodic action of $S_{N}^{+}$, there exists a finite-dimensional unitary representation $v$ of $S_{N}^{+}$ on a Hilbert space $H$ such that
\begin{equation*}
C(X)\cong H\otimes \C\otimes \overline{H} \quad \text{ or } \quad C(X)\cong H\otimes \C^{N}\otimes \overline{H}
\end{equation*}
equivariantly. Commutativity forces $H = \C$, from which the result follows.
\end{rem}

\section{Quantum permutations: the general case}\label{sec:general}

We will now consider arbitrary actions $\alpha : C(X)\to C(X)\otimes C(S_{N}^{+})$ on classical spaces. If $\alpha$ is not assumed to be ergodic, the spectral subspaces need no longer be finite-dimensional, and the previous strategy breaks down. Nevertheless, Theorem~\ref{thm:ergodic} still tells us something about the action because it gives us a description of the \emph{orbits}. Using it, we will be able to describe $X$ and $\alpha$ completely. Before that, let us gather some basic information on the action.

\subsection{Orbit structure of the action}

First, the lack of ergodicity means that the C*-algebra $C(X)^{\alpha}$ is non-trivial. Its spectrum is a compact space $Y$ which can be thought of as the space of orbits of the action. To make this more precise, recall that by the results of \cite{huang2016invariant}, given $x\in X$, there exists a minimal closed subset containing $x$, denoted by $\O(x)$, which is globally invariant under the action in the sense that the ideal $J$ of functions vanishing on $\O(x)$ satisfies $\alpha(J)\subset J\otimes C(S_{N}^{+})$. The set $\O(x)$ is called the \emph{orbit of $x$} and it follows from invariance that $\alpha$ restricts to an action on the quotient $C(X)/J$ which can be proven to be ergodic (see \cite[Thm 4.5]{huang2016invariant}). In particular, by Theorem~\ref{thm:ergodic}, the orbits all consist in either $1$ or $N$ points. Moreover, if the orbit has $N$ points, then the action of $S_{N}^{+}$ is standard, hence so is the corresponding action of the classical permutation group $S_{N}$. In other words, the orbit structure is the same for the action of $S_{N}^{+}$ or for its restriction to $S_{N}$. Let us now use the orbits to obtain information on the spectral subspaces of the action.

\begin{prop}\label{prop:spectrumnonergodic}
Let $X$ be a compact Hausdorff space with an action $\alpha$ of $S_{N}^{+}$. Then, the spectrum of the action is contained in $\{0, 1\}$.
\end{prop}

\begin{proof}
The idea is to restrict to the orbits, for which the result holds by Theorem~\ref{thm:ergodic}. Let $k\geqslant 2$ and let $f\in B_{k}\subset C(X)$. For any $x\in X$, consider its orbit $\O(x)$ with the restricted action
\begin{equation*}
\alpha_{x} : C(\O(x) ) \to C( \O(x) ) \otimes C(S_{N}^{+}).
\end{equation*}
The quotient map $\pi_{x} : C(X)\to C(\O(x))$ is equivariant, hence $B_{k}$ is sent to the spectral subspace of $\alpha_{x}$ for the irreducible representation $\rho_{k}$. But the latter space is $\{0\}$, so we conclude that the restriction of $f$ to $\O(x)$ vanishes. In particular, $f(x) = 0$ and since $x$ was arbitrary, we conclude that $f = 0$.
\end{proof}

Points whose orbit have cardinality one are called \emph{fixed points}. As one may expect, these form a closed subset.

\begin{prop}\label{prop:Zclosed}
The set $Z$ of fixed points is closed.
\end{prop}

\begin{proof}
We will prove that the complement of $Z$ is open. Let $x\in X\setminus Z$, let $x'\in \O(x)\setminus\{x\}$ and let $U, U'$ be open neihgborhoods of $x$ and $x'$ respectively such that $U\cap U' = \emptyset$. Because the orbits of $S_{N}^{+}$ are the same as the orbits of $S_{N}$, there exists a permutation $\sigma$ such that $\sigma(x) = x'$. Since $S_{N}$ acts by homeomorphisms, there exists $V\subset U$ such that $\sigma(V)\subset U'$. But then, $\sigma(V)\cap V = \emptyset$, so that $V$ contains no fixed points. In other words, $V\subset X\setminus Z$, proving that the latter is open.
\end{proof}


\subsection{The fundamental domain}

The inclusion $C(Y) = C(X)^{\alpha}\subset C(X)$ translates into a continuous surjection $X\to Y$ sending each point to its orbit. In particular, the surjection is injective on $Z$ and is $N$ to $1$ on all other points, so that $X$ can be seen, at least set-theoretically, as the quotient of $N$ disjoint copies of $Y$ glued along $Z$. Our goal is to make that identification precise in a topological and equivariant way. It is certainly no surprise to the reader that this will be done through the construction of a sort of fundamental domain for the action, that is to say a ``nice'' subset of $X$ containing exactly one point of every orbit.

Let us start with a ``local'' version of what we need. In the sequel, we denote by $\Stab(x)$ the stabilizer of $x$ under the action of the classical permutation group $S_{N}$. Remember that $S_{N}$ and $S_{N}^{+}$ have the same orbits, hence in particular the same fixed points.

\begin{lem}\label{lem:localsection}
Let $x\in X$ be a point which is not fixed. Then, there exists an open neighborhood $U_{x}$ of $x$ such that for all $\sigma\in S_{N}$, we have
\begin{enumerate}[label=\textup{(\arabic*)}]
\item $\sigma(U_{x}) = U_{x}$, if $\sigma\in \Stab(x)$;
\item $\sigma(U_{x})\cap U_{x} = \emptyset$, otherwise.
\end{enumerate}
\end{lem}

\begin{proof}
The argument is standard, and explained for instance in \cite{kapovich2023note}. We nevertheless give it for completeness. Given $\sigma\notin \Stab(x)$, let $U_{1}$ and $U_{2}$ be disjoint open neighborhoods of $x$ and $\sigma(x)$ respectively. Since the action is by homeomorphisms, there exists an open neighborhood $V$ of $x$ such that $\sigma(V)\subset U_{2}$, hence $V_{\sigma} = U_{1}\cap V$ satisfies $\sigma(V_{\sigma})\cap V_{\sigma} = \emptyset$. Let us now set
\begin{equation*}
V_{x} = \bigcap_{\sigma\in S_{N}\setminus\Stab(x)} V_{\sigma}
\end{equation*}
and observe that it satisfies the second condition of the statement. It then suffices to set
\begin{equation*}
U_{x} = \bigcap_{\sigma\in \Stab(x)}\sigma(V_{x})
\end{equation*}
to get the result.
\end{proof}

We can now bootstrap the previous result to get a fundamental domain for the non-trivial orbits (i.e. those of points which are not fixed). However, in order to do so we will have to exclude the case $N = 6$, for reasons coming from specific features of the classical group $S_{6}$ (we will nevertheless recover the case $N = 6$ in the end by different means, see Theorem~\ref{thm:allactionsSn+}). To make this clearer, we start with a purely group-theoretic result.

\begin{lem}\label{lem:indexNsubgroups}
Let $N\neq 6$ and let $G$ be a subgroup of $S_{N}$ of index $N$. Then, $G$ is the stabilizer of a point in $\{1, \cdots, N\}$ for the standard action.
\end{lem}

\begin{proof}
Let us assume $N > 4$, the other cases can be done by hand. Consider the translation action of $S_{N}$ on the coset space $S_{N}/G$. Since the quotient has cardinality $N$, the action comes from a group homomorphism
\begin{equation*}
\varphi : S_{N}\longrightarrow S_{\vert S_{N}/G\vert} = S_{N}.
\end{equation*}
The only normal subgroup of $S_{N}$ is the alternating group $A_{N}$, which has index $2$ hence cannot be contained in $\ker(\varphi)$. It follows that $\varphi$ is injective, hence surjective. The map $\varphi$ sends $G$ to the stabilizer of a class so that the conclusion holds if $\varphi$ is inner. As is well-known, any automorphism of $S_{N}$ is indeed inner, as soon as $N\neq 6$, hence the result.
\end{proof}

Let us now construct the first part of our fundamental domain.

\begin{prop}\label{prop:fundamentaldomain}
Assume $N\neq 6$. Then, there is an open set $Y'\subset X$ consisting in exactly one point of any non-trivial orbit. Moreover, all the points in $Y'$ have the same stabilizer under the restricted action of $S_{N}$.
\end{prop}

\begin{proof}
The set $\mathcal{V}_{x}$ of all open neighborhoods of $x$ satisfying the conditions of Lemma~\ref{lem:localsection} is partially ordered by inclusion and is inductive, thus by Zorn Lemma it has a maximal element $W_{x}$. We claim that $W_{x}$ satisfies the conclusion of the statement. To see this, let us first consider a fixed point $z\in Z$. If $z\in W_{x}$, then for $\sigma\in S_{N}\setminus\Stab(x)$, $z\in \sigma(W_{x})\cap W_{x}$ contradicting the assumption on $W_{x}$. Therefore, $W_{x}\cap Z = \emptyset$.

Let now $y\in X$ be a point which is not fixed. Then, $\Stab(y)$ has index $N$ in $S_{N}$, so that by Lemma~\ref{lem:indexNsubgroups} it is the stabilizer of a point in $\{1, \ldots, N\}$, in particular this is the stablizer of a point in $\mathcal O (x)$ under the restricted action. Therefore, up to replacing $y$ by another element of $\mathcal O (y)$, we may assume that $\Stab(x) = \Stab(y)$. This means that the open set $W = W_{x}\cup W_{y}$ is fixed by all the elements of $\Stab(x)$. Moreover, for $\sigma\notin\Stab(x)$,
\begin{equation*}
\sigma(W)\cap W = \left(\sigma(W_{x})\cap W_{y}\right)\cup\left(\sigma(W_{y})\cap W_{x}\right).
\end{equation*}
Let $w\in \sigma(W_{x})\cap W_{y}$ and let $\tau\in \Stab(x)$. Then, $\sigma\tau\sigma^{-1}(w) = w$ so that $\sigma\tau\sigma^{-1}(W_{y})\cap W_{y}\neq \emptyset$, and this forces
\begin{equation*}
\sigma\tau\sigma^{-1}\in \Stab(y) = \Stab(x).
\end{equation*}
In other words, $\Stab(x) = \sigma^{-1}\Stab(x)\sigma = \Stab(\sigma^{-1}(x))$. But this is impossible since $\sigma^{-1}(x)$ is another point of the orbit. A similar argument shows that $W_{x}\cap\sigma(W_{y}) = \emptyset$ as well.

Summing up, $W$ satisfies the assumptions of Lemma~\ref{lem:localsection} and by maximality we have $W = W_{x}$. As a consequence, $y\in W_{x}$ and $W_{x}$ contains at least one point of every orbit, and all these points have the same stabilizer. If now $y\in W_{x}$ and $y'\in \O(y)\setminus\{y\}$, any $\sigma\in S_{N}$ such that $\sigma(y) = y'$ is out of $\Stab(x) = \Stab(y)$. As a consequence, $\sigma(W_{x})\cap W_{x} = \emptyset$, leading to $y'\notin W_{x}$. In other words, $W_{x}$ contains exactly one point of any non-trivial orbit.
\end{proof}

To produce a fundamental domain, one simply has to reunite $Y'$ with the fixed points set $Z$. The key fact is that this space is topologically nice because it is compact.

\begin{lem}
Assume $N\neq 6$. The set $\widetilde{Y} = Y'\cup Z$ is closed (hence compact) in $X$.
\end{lem}

\begin{proof}
The closure of $\widetilde{Y}$ is $\overline{Y}'\cup Z$ by Proposition~\ref{prop:Zclosed}. Let $x\in \overline{Y}'$ and assume that $x\notin Y'\cup Z$. In particular, its orbit is non-trivial, hence $Y'$ contains an element $x'\in \O(x)$. Let $\sigma\in S_{N}$ be such that $\sigma(x) = x'$ and let $U$ be an open neighborhood of $x$ such that $\sigma(U)\subset Y'$. Then,
\begin{equation*}
\sigma(U)\cap Y'\subset \sigma(Y')\cap Y'
\end{equation*}
and since the left-hand side is non-empty, we have $\sigma\in \Stab(x)$ by definition of $Y'$. Thus $x = x'\in Y'$, a contradiction.
\end{proof}

\subsection{Classification of arbitrary classical actions}

We are now ready to prove our main result, describing all actions of $S_{N}^{+}$ on compact Hausdorff spaces. To make the statement clear, we first recall the example of non-trivial classical action of $S_{N}^{+}$ construced by H. Huang in \cite{huang2013faithful}.

Let $Y$ be a compact Hausdorff space and let $Z\subset Y$ be closed. The disjoint union $Y^{\coprod N}$ of $N$ copies of $Y$ can be endowed with an action $\widetilde{\alpha}$ of $S_{N}^{+}$ through the identification
\begin{equation*}
C(Y^{\coprod N})\cong C(Y)\otimes \C^{N},
\end{equation*}
the action being trivial on the first tensorand and standard on the second one. One then proves (see \cite[Prop 3.5]{huang2013faithful}) that if $\xi\in \C^{N}$ denotes the all-one vector, then the C*-subalgebra
\begin{equation*}
A = \{f\in C(Y)\otimes \C^{N}\cong C(Y, \C^{N}) \mid f(Z)\subset \C\xi\}
\end{equation*}
is invariant under the action. As a consequence, its spectrum is a compact Hausdorff space with a non-trivial (if $Z\neq Y$) action of $S_{N}^{+}$. It is easy to see that this space is just the quotient of $Y^{\coprod N}$ where all the copies of $Z$ are identified pointwise.

\begin{rem}
If $Y$ is connected, then so is the spectrum of $A$, therefore providing a non-trivial action on a classical connected space. This was the purpose of Huang's construction.
\end{rem}

Let us say that an action of $S_{N}^{+}$ is \emph{of Huang type} if it is equivariantly isomorphic to one obtained through the construction above. We can now state and prove the main result of this work.

\begin{thm}\label{thm:allactionsSn+}
Any non-trivial action of $S_{N}^{+}$ on a compact Hausdorff space $X$ is of Huang type.
\end{thm}

\begin{proof}
We first assume $N\neq 6$. Let $x\in \widetilde{Y}$ and observe that since the action of $S_{N}^{+}$ on $\O(x)$ is standard, so is the action of $S_{N}$. We can therefore, if $x$ is not fixed, label the points of $\O(x)$ as $\{x_{1}, \ldots, x_{N}\}$ in such a way that $\sigma(x_{i}) = x_{\sigma^{-1}(i)}$. There is then by definition a unique $1\leqslant i_{0}\leqslant N$ such that $x_{i_{0}} = x$. Moreover, if $x'\in \widetilde{Y}$ is another point which is not fixed, then it was shown in the proof of Proposition~\ref{prop:fundamentaldomain} that $\Stab(x) = \Stab(x')$. This means that $x' = x'_{i_{0}}$.

Denoting by $\tau_{(ab)}$ the transposition $(ab)$, let us set, for $1\leqslant i\leqslant N$, $\pi_{i} = \tau_{(i_{0}i)\mid \widetilde{Y}} : \widetilde{Y}\to X$ and denote by $Y_{j}$ its range. We can then define a continuous surjection
\begin{equation*}
\widetilde{\pi} : \widetilde{Y}^{\coprod N}\to X
\end{equation*}
by simply setting it to be $\pi_{i}$ on the $i$-th copy of $\widetilde{Y}$. Obviously, $\pi_{i}(z) = \pi_{j}(z)$ for any $z\in Z$ and $1\leqslant i, j\leqslant N$, so that denoting by $\widetilde{X}$ the quotient of of $Y^{\coprod N}$ obtained by identifying all the copies of $Z$ point-wise, $\widetilde{\pi}$ factors through $\pi : \widetilde{X} \to X$. If $q : \widetilde{Y}^{\coprod N} \to \widetilde{X}$ denotes the quotient map, and if $U\subset \widetilde{X}$ is open, then
\begin{equation*}
\pi(U) = \widetilde{\pi}(q^{-1}(U))
\end{equation*}
is open, so that $\pi$ is an open map. Since it is also a bijection, we conclude that it is an homeomorphism. Its inverse yields an isomorphism
\begin{equation*}
\Phi : C(X)\cong C(\widetilde{X}) = \{f\in C(\widetilde{Y})\otimes \C^{N} \mid f(Z)\subset \C\xi\},
\end{equation*}
and all that is left is to prove that $\Phi$ is equivariant, in the sense that
\begin{equation*}
(\Phi\otimes \id)\circ\alpha = (\id\otimes \widetilde{\alpha})\circ\Phi,
\end{equation*}
where $\widetilde{\alpha}$ denotes the Huang action on $C(\widetilde{X})$.

To prove equivariance, we will use once again the orbit structure. If $\pi_{\O(x)} : C(X)\to C(\O(x))$ is the restriction map, then we can identify $C(\O(x))$ with $\C$ or $\C^{N}$ by writing (all $x_{i}$'s are equal to $x$ if the latter is a fixed point)
\begin{equation*}
\pi_{\O(x)}(f) = \sum_{i=1}^{N}f(x_{i})e_{i}.
\end{equation*}
Since we know by Lemma~\ref{lem:standard} that the action of $S_{N}^{+}$ is also standard, and that it restricts to the standard action of $S_{N}$, we conclude that the restricted action $\alpha_{x}$ of $S_{N}^{+}$ on $\O(x)$ is given by the formula
\begin{equation*}
\alpha_{x}\circ\pi_{\O(x)}(f) = \sum_{i=1, j}^{N}f(x_{i})e_{j}\otimes u_{ji}.
\end{equation*}

If $\widetilde{x} = \pi^{-1}(x)$ and $\widetilde{x}_{i} = \tau_{(i_{0}i)}(\widetilde{x})$, then by construction the map $\Phi_{x} : \O(\widetilde{x})\to \O(x)$ sending $\widetilde{x}_{i}$ to $x_{i}$ is $S_{N}$-equivariant. That is, $\alpha_{\widetilde{x}}\circ\Phi_{x} = (\Phi_{x}\otimes \id)\circ\alpha_{x}$. Observing moreover that $\pi_{\O(\widetilde{x})}\circ\Phi = \Phi_{x}\circ\pi_{\O(x)}$, we can now compute
\begin{align*}
\left(\pi_{\O(\widetilde{x})}\otimes \id\right)\circ\left(\Phi\otimes\id\right)\circ\alpha & = (\Phi_{x}\otimes \id)\circ(\pi_{\O(x)}\otimes \id)\circ \alpha \\
& = (\Phi_{x}\otimes \id)\circ\alpha_{x}\circ\pi_{\O(x)} \\
& = \alpha_{\widetilde{x}}\circ\Phi_{x}\circ\pi_{\O(x)} \\
& = \alpha_{\widetilde{x}}\circ\pi_{\O(\widetilde{x})}\circ\Phi \\
& = \left(\pi_{\O(\widetilde{x})}\otimes \id\right)\circ\alpha\circ\Phi
\end{align*}
In particular, for any $f\in C(X)$,
\begin{equation*}
(\mathrm{ev}_{\widetilde{x}}\otimes \id)\left[(\Phi\otimes\id)\circ\alpha(f)\right] = (\mathrm{ev}_{\widetilde{x}}\otimes \id)\left[\alpha\circ\Phi(f)\right].
\end{equation*}
The elements in brackets can be seen as functions from $\widetilde{X}$ to $C(S_{N}^{+})$. We have proven that they coincide when evaluated at $\widetilde{x}$, for any $\widetilde{x}\in \widetilde{X}$, hence they are equal. Since $f$ was arbitrary, the proof is complete.

If now $N = 6$, consider the surjection $\pi_{5} : C(S_{6}^{+})\to C(S_{5}^{+})$ sending $u_{66}$ to $1$. Then, $\alpha' = (\id\otimes \pi_{5})\circ\alpha$ is an action of $S_{5}^{+}$, hence we can write $X$ in the form $Y^{\coprod 5}/Z$ for some closed subset $Z$. Consider now the corresponding action of $S_{6}$ on an orbit of $X$. Restricting it to $S_{5}$ amounts to fixing a point. Therefore, $Z$ contains all the points which are fixed under $S_{6}$, and one point of each non-trivial orbit. In other words, Proposition~\ref{prop:fundamentaldomain} holds for $N = 6$. Since this was the only instance in the first part of the proof where we used the assumption $N\neq 6$, we can now conclude.
\end{proof}

\section{Further examples}\label{sec:further}

The quantum permutation groups belong to a class of compact quantum groups called \emph{easy quantum groups} and introduced in \cite{banica2009liberation}. The common point of these objects is that they are defined using the combinatorics of partitions of finite sets. As a consequence, they share many common features and it is quite natural to try to extend the preceding results to other examples in this class. We will here consider the ``free'' quantum groups, since they should exhibit some rigidity with respect to classical actions comparable to what happens for $S_{N}^{+}$. We will not give the general definition of easy quantum groups here, but rather work case-by-case.

\subsection{Free orthogonal quantum group}

The free orthogonal quantum groups $O_{N}^{+}$ were introduced by S.Z. Wang in \cite{wang1995free}, to which we refer for the definition. It turns out that classical actions are much easier to study in that case.

\begin{thm}\label{thm:ergodicon+}
The free orthogonal quantum group $O_{N}^{+}$ cannot act non-trivially on a compact Hausdorff space for $N\geqslant 3$.
\end{thm}

\begin{proof}
Consider first an ergodic action of $O_{N}^{+}$ on a compact Hausdorff space $X$. As explained in the proof of \cite[Thm 6.7]{freslon2021tannaka}, the same arguments as for $S_{N}^{+}$ show that if $k$ is the least integer such that $B_{k}\neq \{0\}$, then $B_{0}\oplus B_{k}$ is a finite-dimensional C*-subalgebra. Therefore, there is a corresponding action of $O_{N}^{+}$ on a finite space of dimension $d$. The universal property of $S_{d}^{+}$ then implies that the action is of the form $(\id\otimes \pi)\circ\alpha_{d}$ for some $*$-homomorphism $\pi : C(S_{d}^{+})\to C(O_{N}^{+})$. Now, $\pi(C(S_{d}^{+}))$ is of the form $C(\HH)$ for some compact quantum group $\HH$, and it is well-known that $\HH$ is either trivial or equals $O_{N}^{+}$ or its projective version $PO_{N}^{+}$. But $C(\HH)$ is by definition generated by projections, hence cannot be equal to $C(O_{N}^{+})$ since the latter surjects onto $C(O_{N})$ and $O_{N}$ is not totally disconnected. Similarly, $C(PO_{N}^{+})$ is not generated by projections and we conclude that $\pi(C(S_{N}^{+})) = \C$. In other words, the action on $B_{0}\oplus B_{k}$ is trivial and ergodic, hence $B_{k} = \{0\}$ which by definition means that the original action is trivial.

Consider now an arbitrary action, let $x\in X$ and consider the restriction of the action to its orbit $\O(x)$. This is ergodic, hence by the first part of the proof, $\O(x) = \{x\}$, meaning that $\{x\}$ is fixed. Since $x$ was arbitrary, the action is trivial.
\end{proof}

\begin{rem}
For $N = 1$, we have $O_{1}^{+} = \Z_{2}$, which does have many actions on classical spaces, since these correspond to order $2$ homeomorphisms. For $N = 2$, $O_{2}^{+}$ is isomorphic to $SU_{-1}(2)$. There is an inclusion $C(SO(3))\subset C(SU_{-1}(2))$ such that $\D(C(SO(3)))\subset C(SO(3))\otimes C(SO(3))\subset C(SO(3))\otimes C(SU_{-1}(2))$, therefore yielding a classical action of $O_{2}^{+}$ on $SO(3)$. In other words, Theorem~\ref{thm:ergodicon+} fails for $N\leqslant 2$.
\end{rem}

We also get for free the same result for another easy quantum group associated to non-crossing partition.

\begin{cor}\label{cor:bistochastic}
The free bistochastic quantum groups $B_{N}^{+}$ cannot act non-trivially on a compact Hausdorff space for $N\geqslant 4$.
\end{cor}

\begin{proof}
This follows from Theorem~\ref{thm:ergodicon+} since $B_{N}^{+}$ is known to be isomorphic to $O_{N-1}^{+}$.
\end{proof}

Even though it is not an easy quantum group, we take to occasion to settle the case of the quantum automorphism group of $M_{N}(\C)$ preserving the canonical trace (see \cite{banica1999symmetries} for the definition), because it follows directly from what we have done.

\begin{cor}
The quantum automorphism group $\G$ of $(M_{N}(\C), \mathrm{tr})$ does not act non-trivially on a compact Hausdorff space for $N\geqslant 3$.
\end{cor}

\begin{proof}
It is well-known that there is a quantum group embedding $\imath : C(\G)\hookrightarrow C(O_{N}^{+})$, so that any non-trivial action $\alpha : C(X)\to C(X)\otimes C(\G)$ induces a non-trivial action $(\id\otimes \imath)\circ\alpha$ of $O_{N}^{+}$, hence the result by Theorem~\ref{thm:ergodicon+}.
\end{proof}

Once again, the result fails for $N = 2$ since $\G$ is then the classical group $SO(3)$.

\subsection{Direct products}

The list of easy quantum groups associated to non-crossing partitions contains, besides the examples above, some direct products with the cyclic group of order two $\Z_{2}$. It is quite intuitive that the actions of a direct product of two compact quantum groups should correspond to pairs of actions of the factors which commute in a certain sense. Proposition~\ref{prop:directproduct} below makes this precise and, even though it is certainly well-known to experts, we give a proof for completeness. We start with a convenient definition:

\begin{de}
Two actions $\alpha$ and $\beta$ of two compact quantum groups $\G$ and $\HH$ respectively on the same C*-algebra $A$ are said to be \emph{commuting} if
\begin{equation*}
(\alpha\otimes \id)\circ\beta = (\id\otimes \Sigma)\circ(\beta\otimes\id)\circ\alpha
\end{equation*}
as maps from $A$ to $A\otimes C(\G)\otimes C(\HH)$, where $\Sigma$ denotes the flip map.
\end{de}

The result that we need is a correspondance between commuting pairs of actions and actions of the direct product. Recall that given two compact quantum groups $\G$ and $\HH$, their \emph{direct product} $\G\times \HH$ is the compact quantum group with function algebra $C(\G\times \HH) = C(\G)\otimes C(\HH)$ and with coproduct
\begin{equation*}
\D_{\G\times \HH} = (\id\otimes \Sigma\otimes \id)\circ\left(\Delta_{\G}\otimes \Delta_{\HH}\right)
\end{equation*}
(see \cite{wang1995tensor} for details). Assuming that we are working with the universal C*-algebras (which we will do hereafter since there is no loss of generality), the factors can be recovered as quantum subgroups through the maps
\begin{align*}
\pi_{\G} = (\id\otimes \varepsilon_{\HH}) : C(\G\times \HH)\to C(\G) \quad \text{ and } \quad \pi_{\HH} = (\varepsilon_{\G}\otimes \id) : C(\G\times \HH)\to C(\HH),
\end{align*}
where $\varepsilon$ denotes the counit. Notice that, using the equality $(\varepsilon\otimes \id)\circ\D = \id = (\id\otimes \varepsilon)\circ\D$, we get
\begin{align*}
(\pi_{\G}\otimes\pi_{\HH})\circ\D_{\G\times \HH} = \id_{C(\G\times \HH)}.
\end{align*}

\begin{prop}\label{prop:directproduct}
Let $\G$ and $\HH$ be two compact quantum groups and let $A$ be a C*-algebra. If $\alpha$ and $\beta$ are commuting actions of $\G$ and $\HH$ respectively on $A$, then
\begin{equation*}
\gamma = (\alpha\otimes\id)\circ\beta
\end{equation*}
is an action of $\G\times \HH$ on $A$. Conversely, if $\gamma$ is an action of $\G\times \HH$ on $A$, then $\alpha = (\id\otimes \pi_{\G})\circ\gamma$ and $\beta = (\id\otimes \pi_{\HH})\circ\gamma$ are commuting actions of $\G$ and $\HH$ respectively on $A$. Moreover, this is yields a one-to-one correspondence.
\end{prop}

\begin{proof}
Let $\alpha$ and $\beta$ be commuting actions of $\G$ and $\HH$ respectively. We have to prove that $\gamma = (\alpha\otimes\id)\circ\beta$ is an action. Condition \eqref{cond:coaction} in Definition~\ref{def:action} follows from an elementary computation:
\begin{align*}
(\id\otimes \Delta_{\G\times \HH})\circ\gamma & = (\id\otimes \id\otimes \Sigma\otimes \id)\circ(\id\otimes \Delta_{\G}\otimes \Delta_{\HH})\circ(\alpha\otimes \id)\circ\beta \\
& = (\id\otimes\id\otimes \Sigma\otimes \id)\circ(\alpha\otimes \id\otimes \id\otimes \id)\circ(\alpha\otimes \id\otimes \id)\circ(\beta\otimes \id)\circ\beta \\
& = (\alpha\otimes \Sigma\otimes \id\otimes \id)\circ(\alpha\otimes \id\otimes \id)\circ(\beta\otimes \id)\circ\beta \\
& = (\alpha\otimes \id\otimes \id\otimes \id\otimes \id)\circ\left(\left[(\id\otimes \Sigma)\circ(\alpha\otimes \id)\circ\beta\right]\otimes \id\right)\circ\beta \\
& = (\alpha\otimes \id\otimes \id\otimes \id\otimes \id)\circ\left(\left[(\beta\otimes \id)\circ\alpha\right]\otimes \id\right)\circ\beta \\
& = (\alpha\otimes \id\otimes \id\otimes \id)\circ(\beta\otimes \id\otimes \id)\circ(\alpha\otimes \id)\circ\beta \\
& = (\gamma\otimes \id)\circ\gamma.
\end{align*}
To see that Condition \eqref{cond:density} holds, simply observe that $(\id\otimes \varepsilon_{\G}\otimes \varepsilon_{\HH})\circ\gamma = \id_{A}$, and this is known to imply Condition \eqref{cond:density} (see for instance the proof of \cite[Lem 2.13]{DC17}).

Conversely, let $\gamma$ be an action of $\G\times \HH$. By construction, $\alpha$ and $\beta$ are actions of $\G$ and $\HH$ respectively. Moreover, using the equality $(\varepsilon\otimes \id)\circ\D = \id = (\id\otimes \varepsilon)\circ\D$, we get
\begin{align*}
(\alpha\otimes \id)\circ\beta & = (\id\otimes \pi_{\G}\otimes \id)\circ(\gamma\otimes \id)\circ(\id\otimes \pi_{\HH})\circ\gamma \\
& = (\id\otimes \pi_{\G}\otimes \pi_{\HH})\circ(\gamma\otimes \id)\circ\gamma \\
& = (\id\otimes \pi_{\G}\otimes \pi_{\HH})\circ(\id\otimes \Delta_{\G\times \HH})\circ\gamma \\
& = \gamma.
\end{align*}
A similar calculation yields $(\id\otimes \Sigma)\circ(\beta\otimes \id)\circ\alpha = \gamma$, hence $\alpha$ and $\beta$ are commuting.
\end{proof}

Let us now go back to easy quantum groups and consider two more examples. As one may expect, any action commutes with the trivial action. Therefore, an action of $B_{N}^{\prime +} = B^{+}_{N}\times \Z_{2}$ on a classical compact space is the same by Corollary~\ref{cor:bistochastic} as an action of $\Z_{2}$, that is to say an order two homeomorphism. The other case is $S_{N}^{\prime +} = S_{N}^{+}\times \Z_{2}$. Using our results, we can once again describe them completely.

\begin{prop}
The actions of $S_{N}^{\prime +}$ on classical compact spaces are exactly the actions of Huang type on $Y^{\coprod N}/Z$, where $Y$ is moreover equipped with an action of $\Z_{2}$ which fixes $Z$ globally.
\end{prop}

\begin{proof}
Let us denote by $\omega$ the generator of $\Z_{2}$. Let $X = Y^{\coprod N}/Z$ be equipped with its Huang type action $\alpha$ of $S_{N}^{+}$. An action of $\Z_{2}$ on $Y$ is given by an order $2$ homeomorphism $g\in C(Y)$ and the element $\widetilde{g} = (g, \cdots, g)\in C(Y^{\coprod N})$ is an order $2$ homeomorphism. Since it fixes $Z$ globally, it factors through an order $2$ homeomorphism of $X$ which will still be denoted by $\widetilde{g}$ and provides an extension of the action to $X$. Denoting that action by $\beta$, we have, for $f\in C(X)$,
\begin{equation*}
\beta(f) = f\otimes 1 + (f\circ \widetilde{g})\otimes \omega.
\end{equation*}
We now simply check that $\alpha$ and $\beta$ commute:
\begin{align*}
(\alpha\otimes \id)\circ\beta\left(\sum_{i=1}^{N}f_{i}\otimes e_{i}\right) & = \sum_{i, j=1}^{N}f_{i}\otimes e_{j}\otimes u_{ij}\otimes 1 + \sum_{i, j=1}^{N}(f_{i}\circ \widetilde{g})\otimes e_{j}\otimes u_{ij}\otimes \omega \\
& = (\id\otimes \Sigma)\left(\sum_{i, j=1}^{N}f_{i}\otimes e_{j}\otimes 1\otimes u_{ij} + \sum_{i, j=1}^{N}(f_{i}\circ \widetilde{g})\otimes e_{j}\otimes \omega\otimes u_{ij}\right) \\
& = (\id\otimes \Sigma)\left(\sum_{i, j = 1}^{N}\beta(f_{i})\otimes u_{ij}\right) \\
& = (\id\otimes \Sigma)\circ(\beta\otimes \id)\circ\alpha.
\end{align*}

Conversely, let $\alpha$ and $\beta$ be two commuting actions of $S_{N}^{+}$ and $\Z_{2}$ respectively on $X$. By Theorem~\ref{thm:allactionsSn+}, we know that $X$ is of Huang type. Moreover, let $\E_{\alpha} = (\id\otimes h_{S_{N}^{+}})\circ\alpha$ be the conditional expectation onto the functions which are constant on orbits. Then,
\begin{align*}
(\E_{\alpha}\otimes \id)\circ\beta & = (\id\otimes h_{S_{N}^{+}}\otimes \id)\circ(\alpha\otimes \id)\circ\beta \\
& = (\id\otimes h_{S_{N}^{+}}\otimes \id)\circ(\id\otimes \Sigma\otimes \id)\circ(\beta\otimes \id)\circ\alpha \\
& = (\id\otimes \id\otimes h_{S_{N}^{+}})\circ(\beta\otimes \id)\circ\alpha \\
& = \beta\circ\E_{\alpha}.
\end{align*}
This means that $C(X)^{\alpha}\simeq C(Y)$ is invariant under $\beta$, i.e. that there is a restricted action of $\Z_{2}$ on $Y$.

Let now $J\subset C(X)$ be a closed ideal which is invariant under $\alpha$ in the sense that $\alpha(J)\subset J\otimes 1$. Then,
\begin{align*}
\alpha\left[\left(\id\otimes\mathrm{ev}_{\omega}\right)\circ\beta(J)\right] & = (\id\otimes \mathrm{ev}_{\omega}\otimes \id)\circ(\id\otimes \Sigma)\circ(\beta\otimes \id)\circ\alpha(J) \\
& \subset (\id\otimes \id\otimes\mathrm{ev}_{\omega})\circ(\id\otimes \Sigma)\circ(\beta\otimes \id)(J\otimes 1) \\
& = \left[(\id\otimes \mathrm{ev}_{\omega})\circ\beta(J)\right]\otimes 1.
\end{align*}
Since $(\id\otimes \mathrm{ev}_{\omega})\circ\beta(f) = f\circ\widetilde{g}$, it follows that $\omega(J)$ is invariant as soon as $J$ is. This induces an isomorphism between the corresponding quotients so that if $J$ is the ideal of functions vanishing at a fixed point $z$, then $\omega(J)$ has codimension $1$, hence is also the ideal of functions vanishing at a fixed point $z'$. In other words, the action of $\Z_{2}$ fixes $Z$ globally.
\end{proof}

\subsection{Hyperoctahedral quantum group}

In this section, we will consider the hyperoctaedral quantum groups $H_{N}^{+}$ introduced in \cite{banica2007hyperoctahedral}. Let us start with the definition.

\begin{de}
Let $C(H_{N}^{+})$ be the universal C*-algebra generated by elements $(u_{ij})_{1\leqslant i, j\leqslant N}$ such that
\begin{itemize}
\item $u_{ij}^{*} = u_{ij}$ for all $1\leqslant i, j\leqslant N$;
\item $u_{ij}^{2}$ is a projection for all $1\leqslant i, j\leqslant N$;
\item $\displaystyle\sum_{k=1}^{N} u_{ik}^{2} = 1 = \sum_{k=1}^{N}u_{kj}^{2}$ for all $1\leqslant i, j\leqslant N$.
\end{itemize}
The map
\begin{equation*}
\Delta : u_{ij}\mapsto \sum_{k=1}^{N}u_{ik}\otimes u_{kj}
\end{equation*}
extends to a $*$-homomorphism yielding a compact quantum group called the \emph{hyperoctahedral quantum group} and denoted by $H_{N}^{+}$.
\end{de}

The case of quantum reflection groups is more involved than the previous ones, but we will take advantage of the fact that $H_{N}^{+}$ turns out to be a ``quantum permutation subgroup'' in disguise. Indeed, let us set $v_{ij} = u_{ij}^{2}$ and consider the following matrix
\begin{equation*}
\left(\arraycolsep=1.4pt\def\arraystretch{2.2}\begin{array}{cc}
\displaystyle\frac{v-u}{2} & \displaystyle\frac{v+u}{2} \\
\displaystyle\frac{v+u}{2} & \displaystyle\frac{v-u}{2}
\end{array}\right)
\end{equation*}
Observing that $(u_{ij} \pm v_{ij})/2$ is a projection, we see that this satisfies the defining relations of $C(S_{2N}^{+})$, hence there is a surjective $*$-homomorphism $\pi : C(S_{2N}^{+})\to C(H_{N}^{+})$. Using it, we get an action of $H_{N}^{+}$ on $\C^{2N}$ from the standard action of $C(S_{2N}^{+})$, namely
\begin{equation*}
\alpha(e_{i}) = \left\{\begin{array}{ccc}
\displaystyle\sum_{j=1}^{N}e_{j}\otimes \frac{v_{ji} - u_{ji}}{2} + \displaystyle\sum_{j=N+1}^{2N}e_{j}\otimes \frac{v_{ji} + u_{ji}}{2}, & \text{ if } & 1\leqslant i\leqslant N; \\
\displaystyle\sum_{j=1}^{N}e_{j}\otimes \frac{v_{ji} + u_{ji}}{2} + \displaystyle\sum_{j=N+1}^{2N}e_{j}\otimes \frac{v_{ji} - u_{ji}}{2}, & \text{ if } & N+1\leqslant i\leqslant 2N. \\
\end{array}\right.
\end{equation*}
This will be called the \emph{standard action} of $H_{N}^{+}$ on $\C^{2N}$ and will serve as our building block, in a similar way to what was done for quantum permutations.

We will also need some information about the representation theory of $H_{N}^{+}$, that we now recall. It was proven in \cite{banica2009fusion} that the irreducible representations of $H_{N}^{+}$ can be indexed by words on $\{0, 1\}$ with $\emptyset$ being the trivial representation and $1$ being the fundamental one. We will denote by $\mu_{w}$ the representation indexed by the word $w$. If $X$ is a compact space on which $H_{N}^{+}$ acts, it follows from \cite{freslon2021tannaka} that
\begin{equation*}
B = B_{\emptyset}\oplus B_{0}\oplus B_{1}
\end{equation*}
is an invariant subalgebra, so that there is an equivariant quotient map $\pi : X\to Y$ with $Y$ finite with spectrum contained in $\{\emptyset, 0, 1\}$. This suggests to try to parallel the previous results concerning $S_{N}^{+}$ by focusing on these three irreducible representations. Before doing that, let us recall a few facts. First, the coefficients of the irreducible representation $\mu_{0}$ generate a copy of $C(S_{N}^{+})$ inside $C(H_{N}^{+})$ and the restriction of the coproduct endows it with the compact quantum group structure of $S_{N}^{+}$. Second, it turns out that using the previous notations, $v$ is a representation equivalent to $\mu_{0}\oplus \mu_{\emptyset}$. As for the case of $S_{N}^{+}$, a crucial element of the proof will be the restricted action of the classical hyperoctahedral group $H_{N}$.

\begin{lem}\label{lem:hyperoctahedraltransitive}
Let $\alpha$ be an ergodic action of $H_{N}^{+}$ with spectrum contained in $\{\emptyset, 0, 1\}$. Then, the restricted action of $H_{N}$ is transitive.
\end{lem}

\begin{proof}
The spectrum of the restricted action is contained in the restriction of the spectrum, hence it is enough to prove that the restrictions of $\mu_{0}$ and $\mu_{1}$ to $H_{N}$ are irreducible. For $\mu_{0}$, it is already irreducible when restricted to $S_{N}\subset H_{N}$ since it corresponds to the standard representation on $\C^{N-1}$. As for $\mu_{1}$, it is the representation of $H_{N}$ on $\C^{N}$ by signed permutation matrices. When restricted to $S_{N}$, that representation has two subrepresentations: the one given by the all-one vector $\xi$ and its orthogonal complement. Since $\xi$ is not fixed by all signed permutation matrices, we conclude that it is not a subrepresentation for $H_{N}$, hence the restriction of $\mu_{1}$ is irreducible.
\end{proof}

\begin{lem}\label{lem:hyperoctahedral_2N}
Let $N\geqslant 4$. Then, any ergodic action of $H_{N}^{+}$ on $\C^{2N}$ is either isomorphic to the standard one or induced from $S_{N}^{+}$.
\end{lem}

\begin{proof}
Let $\alpha$ be such an action. By the universal property of $C(S_{2N}^{+})$, there exists a $*$-homomorphism $\pi : C(S_{2N}^{+})\to C(H_{N}^{+})$ such that $\alpha = (\id\otimes \pi)\circ\alpha_{2N}$. Let us set, for $1\leqslant i, j\leqslant N$,
\begin{equation*}
a_{ij} = \pi(p_{i+N, j}) - \pi(p_{ij}); \quad b_{ij} = \pi(p_{i+N, j}) + \pi(p_{ij}); \quad f_{i} = e_{i} - e_{i+N}; \quad g_{i} = e_{i} + e_{i+N}.
\end{equation*}
Then,
\begin{equation*}
\alpha(f_{i}) = \sum_{j=1}^{N}f_{j}\otimes a_{ji} \quad\text{ and }\quad \alpha(g_{i}) = \sum_{j=1}^{N}g_{j}\otimes b_{ji}
\end{equation*}
and the fact that $\alpha$ is an action implies that $a$ and $b$ are representations of $H_{N}^{+}$. Moreover, the elements $a_{ij}^{2} = b_{ij}$ are projections summing up to $1$ on rows and columns. As a consequence, the representation $b$ is either a sum of $N$ copies of the trivial representation, or equivalent to $\mu_{0}\oplus\mu_{\emptyset}$ (any other representation has dimension at least $\dim(\mu_{00}) = (N-1)^{2} - N$ which is greater than $N$ if $N\geqslant 4$). As for $a$, it is either one of the two preceding possibilities, or equivalent to $\mu_{1}$. If $a$ is trivial, so is $b$ hence the action itself is trivial which is prevented by ergodicity. If $a\sim\mu_{0}\oplus\mu_{\emptyset}$, then all the coefficients belong to $C(S_{N}^{+})\subset C(H_{N}^{+})$, hence the action is not ergodic. As a conclusion, $a\sim \mu_{1}$. This implies that $\pi$ is onto and that the action is isomorphic to the standard one.
\end{proof}

Using this, we can study more generally actions with prescribed spectrum. To do so, let us recall that since the generators of $C(S_{N}^{+})$ satisfy all the defining relations of $C(H_{N}^{+})$, there is a surjective $*$-homomorphism $C(H_{N}^{+})\to C(S_{N}^{+})$. Composing an action with it yields the \emph{restriction} of that action to $S_{N}^{+}$. We will also use the well-known fact that the classical hyperoctahedral group $H_{N}$ is isomorphic to the semi-direct product $\Z_{2}^{N}\rtimes S_{N}$, where $S_{N}$ acts by permutation of the copies.

\begin{prop}\label{prop:standardhyperoctahedral}
Let $X$ be a (necessarily finite) space with an ergodic action of $H_{N}^{+}$ whose spectrum is contained in $\{\emptyset, 0, 1\}$. Then, $X$ is either trivial, standard or induced from the standard action of $S_{N}^{+}$.
\end{prop}

\begin{proof}
Let $\alpha$ be such an action and consider its restriction to $S_{N}^{+}$. If it is ergodic, then $X$ has $N$ points and the spectrum of the restriction is $\{0, 1\}$ (these are the labels of irreducible representations of $S_{N}^{+}$). Since $\mu_{\emptyset}\mapsto \rho_{0}$, $\mu_{0}\mapsto \rho_{1}$ and $\mu_{1}\mapsto \rho_{1}$, we can conclude that $\mu_{0}$ has mutiplicity $1$ and $\mu_{1}$ has multiplicity $0$ in the original action. This implies that the action is induced from $S_{N}^{+}$ in the sense that
\begin{equation*}
\alpha(C(X))\subset C(X)\otimes C(S_{N}^{+})\subset C(X)\otimes C(H_{N}^{+}).
\end{equation*}

In general now, let us first look at the subspace $A = B_{\emptyset}\oplus B_{0}\subset C(X)$. Because of the fusion rules of $H_{N}^{+}$, $A$ is a subalgebra. Moreover, by definition of the spectral subspaces, $\alpha(A)\subset A\otimes C(S_{N}^{+})$ so that we have an action of $S_{N}^{+}$ on the spectrum of $A$. Since $\dim(B_{\emptyset}) = 1$, this is an ergodic action and we conclude by Theorem~\ref{thm:ergodic} that the spectrum has $N$ points. Since $u^{0}$ has dimension $N-1$, this implies that it has multiplicity one. Therefore, $X$ itself has $(m+1)N$ points, where $m$ is the multiplicity of $u^{1}$.

Moreover, the action restricted to $S_{N}^{+}$ is a disjoint union of $m+1$ standard actions, and so is the restriction to $S_{N}$. We can therefore write $X$ as the disjoint union of subsets $(X^{i})_{1\leqslant i\leqslant m+1}$ on which $S_{N}$ acts standardly. Let us fix a numbering $X^{1} = (x^{(1)}_{1}, \cdots, x^{(1)}_{N})$ such that $\sigma(x^{(1)}_{i}) = x^{(1)}_{\sigma^{-1}(i)}$ and let $\omega_{1}, \cdots, \omega_{N}$ be the generators of the factors of $\Z_{2}^{N}$ in $H_{N}^{+}$ (which acts by restriction on $X$). Then, there exists $\alpha, \beta$ such that
\begin{equation*}
\omega_{1}(x^{(1)}_{N}) = x^{(\alpha)}_{\beta}.
\end{equation*}
Moreover, for any $1 < i\leqslant N$, we have
\begin{equation*}
\omega_{1}(x^{(1)}_{i}) = \omega_{1}\tau_{iN}(x^{(1)}_{N}) = \tau_{iN}\omega_{1}(x^{(1)}_{N}) = x^{(\alpha)}_{\tau_{iN}(\beta)}
\end{equation*}
and more generally for any $i\neq k$,
\begin{equation*}
\omega_{k}(x^{(1)}_{i}) = \tau_{1k}\omega_{1}\tau_{1k}(x^{(1)}_{i}) = \left\{\begin{array}{ccc}
x^{(\alpha)}_{\tau_{1k}\tau_{iN}(\beta)}, & \text{ if } & i\neq 1; \\
x^{(\alpha)}_{\tau_{1k}\tau_{kN}(\beta)}, & \text{ if } & i = 1.
\end{array}\right.
\end{equation*}
Similarly, we can write $\omega_{1}(x^{(1)}_{1}) = x^{(\gamma)}_{\delta}$ and it then follows that
\begin{equation*}
\omega_{k}(x^{(1)}_{k}) = \tau_{1k}\omega_{1}\tau_{1k}(x^{(1)}_{k}) = x^{(\gamma)}_{\tau_{1k}(\delta)}.
\end{equation*}
It follows from these equalities that the orbit of $X^{1}$ under the action of $\Z_{2}^{N}$ is contained in $X^{1}\cup X^{\alpha}\cup X^{\gamma}$. Since the action of $S_{N}$ preserves each of these three sets, we conclude from the transitivity of the action of $H_{N}$ proven in Lemma~\ref{lem:hyperoctahedraltransitive} that $m+1\leqslant 3$.


Let us now consider the stabilizer $G$ of $x_{1}^{(1)}$ under the restricted action of $H_{N}$. Because the restricted action of $H_{N}$ is transitive by Lemma~\ref{lem:hyperoctahedraltransitive}, the index of $G$ is $(m+1)N$. On the other hand, it follows from the calculations above that $\omega_{i}\omega_{j}$ fixes $x_{1}^{(1)}$ for all $2\leqslant i < j\leqslant N$. We claim that these elements generate a subgroup $K$ of $\Z_{2}^{N}$ isomorphic to $\Z_{2}^{N-2}$. To prove this, first note that they generate the same group as the elements $\widehat{\omega}_{i} = \omega_{i}\omega_{i+1}$ for $2\leqslant i\leqslant N-2$. This is a finite abelian group with all elements of ordre $2$, hence isomorphic to $\Z_{2}^{M}$. Adding $\omega_{2}$ to it yields the subgroup generated by $\omega_{2}, \cdots, \omega_{N}$, that is to say $\Z_{2}^{N-1}$. Therefore, $M = N-2$.

Now, $G$ contains $K$ and a subgroup of $S_{N}$ isomorphic to $S_{N-1}$, which together generate a subgroup of index $4N$ in $H_{N}$. Therefore, the index of $G$ divides $4N$, i.e. $(m+1)N$ divides $4N$. Since $m+1\leqslant 3$, we must have $m+1 \leqslant 2$, hence $m \leqslant 1$ and the proof is complete.
\end{proof}

We now have the tools to classify classical actions of $H_{N}^{+}$ in the same fashion as for $S_{N}^{+}$.

\begin{thm}\label{thm:ergodichn+}
An ergodic action of $H_{N}^{+}$ on a compact Hausdorff space is either trivial, standard or induced from the standard action of $S_{N}^{+}$.
\end{thm}

\begin{proof}
It was proven in \cite[Thm 6.7]{freslon2021tannaka} that given a word $\w$ on $\{0, 1\}$, $B_{\w}\cdot B_{\w}\subset B_{\emptyset}\oplus B_{0}\oplus B_{1} = B'$. However, there are two ingredients missing to prove that $B'\oplus B_{\w}$ is a C*-algebra. One is the stability under multiplication of $B_{\w}$ by $B'$, and the other is the stability under involution (recall that in general, the representation $\mu_{\w}$ is not self-conjugate). Some careful work is therefore needed, and for the sake of clarity, we will split the proof into three steps.

\begin{enumerate}
\item The first step is to prove that $B'\oplus B_{\w}$ is an algebra if the word $\w$ is not constant (i.e. different from $0^{k}$ and $1^{k}$). To do this, let $x\in B_{0}$ and $y\in B_{\w}$. Then, denoting by $\w_{>1}$ (resp. $\w_{<n}$) the subword of $\w$ made of all its letters except for the first one (resp. the last one), we have
\begin{eqnarray*}
x\cdot y & \in & B_{0\w}\oplus B_{\w} \oplus \delta_{\w_{1}, 0}B_{\w_{>1}}, \\
y\cdot x & \in & B_{\w0}\oplus B_{\w} \oplus \delta_{\w_{n}, 0}B_{\w_{<n}}.
\end{eqnarray*}
Because $C(X)$ is commutative, the two elements on the left are equal. Comparing lengths and observing that $0\w = \w0$ if and only if $\w$ is constant and similarly for $\w_{>1} = \w_{<n}$, we conclude that $x\cdot y\in B_{\w}$. If now $x\in B_{1}$, we have
\begin{eqnarray*}
x\cdot y & \in & B_{1\w}\oplus B_{\overline{\w}_{1}\w_{>1}}\oplus \delta_{\w_{1},1}B_{\w_{>1}}, \\
y\cdot x & \in & B_{\w1}\oplus B_{\w_{<n}\overline{\w}_{n}}\oplus \delta_{\w_{n},1}B_{\w_{<n}},
\end{eqnarray*}
and the same arguments show that, because $\w$ is not constant, $x\cdot y = 0 = y\cdot x$, which together with what precedes gives $B'\cdot B_{\w}\subset B_{\w}$.
\item We now turn to constant words, to prove once again that $B'\oplus B_{\w}$ is an algebra. Let us consider the following subspace of $C(X)$:
\begin{equation*}
C(X)_{0} = \bigoplus_{k\in \N}B_{0^{k}}.
\end{equation*}
The fusion rules of $H_{N}^{+}$ imply that this is a C*-subalgebra, and that $\alpha(C(X)_{0})\subset C(X)_{0}\otimes C(S_{N}^{+})$. Because the representation $\mu_{0^{k}}$ corresponds to the representation $\rho_{k}$ of $S_{N}^{+}$, we conclude by Theorem~\ref{thm:ergodic} that $B_{0^{k}} = \{0\}$ if $k\geqslant 1$, hence the claim for constant words with letter $0$. As for words of the form $1^{k}$, the only non-trivial fact to check is that $B_{1}\cdot B_{1^{k}}\subset B_{1^{k}}$. But this can be done exactly as for $S_{N}^{+}$ in the proof of Lemma~\ref{lem:ergodicactionspectrumk}. Indeed, observe that when defining the projection $P_{\vert^{\odot k}}^{H_{N}^{+}}$, we substract the supremum of a family of projections which is contained in that used for $P_{\vert^{\odot k}}^{S_{N}^{+}}$. As a consequence, the former projection is dominated by the latter, so that their ranges are contained in one another. Otherwise said, the vectors $\eta_{1}$, $\eta_{2}$ constructed in the proof of Lemma~\ref{lem:ergodicactionspectrumk} also lie in the carrier spaces of the representations $u_{\vert^{\odot k-1}}$ and $u_{\vert^{\odot k}}$ for $H_{N}^{+}$. As a consequence, the argument works the same and we conclude that $B'\oplus B_{1^{k}}$ is an algebra. 
\item We are now ready to prove that the spectrum of the action is contained in $\{\emptyset, 0, 1\}$. We first assume that $\w$ is a \emph{symmetric} word -- in the sense that $\overline{\w} = \w$ -- of length at least $2$. Then $B'\oplus B_{\w}$ is a C*-subalgebra such that $B_{\w}^{2}\subset B'$ and the same argument as in the proof of Theorem~\ref{thm:ergodic} then shows that $B_{\w} = \{0\}$. If $\w$ is not symmetric, we consider $\widetilde{B} = B'\oplus B_{\w}\oplus B_{\overline{\w}}$, which is stable under the involution. To prove that this is a subalgebra, we only need to prove that $B_{\w}\cdot B_{\overline{\w}}\subset \widetilde{B}$. If $x\in B_{\w}$ and $y\in B_{\overline{\w}}$, then $x\cdot y$ belongs to a sum of spaces of the form $B_{\w'}$ where $\w'$ is a symmetric subword of $\w\overline{\w}$. Since we just proved that these spaces are trivial for words of length at least $2$, we conclude that $x\cdot y\in B'$. This shows that $\widetilde{B}$ is a subalgebra and we conclude as before that $B_{\w} = \{0\}$. Summing up, the only non-trivial spectral subspaces are $B_{\emptyset}$, $B_{0}$ and $B_{1}$ and we can now conclude by Proposition~\ref{prop:standardhyperoctahedral}.
\end{enumerate}
\end{proof}

Based on this result, one may try to describe all classical actions of $H_{N}^{+}$ using the orbit decomposition in the same way as for $S_{N}^{+}$. However, things are more complicated here. Indeed, in the case of $S_{N}^{+}$, we had only two possible orbits, of size either one or $N$. Moreover, we saw that we could somehow ``collapse'' some orbits of size $N$ to trivial orbits, and that this exhausts all possibilites. For $H_{N}^{+}$, we have three possible orbits and it is therefore natural to expect that orbits of size $2N$ can collapse either to trivial orbits or to orbits of size $N$. We will now show that this is indeed the case.

\begin{ex}
Let $Y$ be a compact space and let $Z_{1}\subset Y$ be a closed subset. We consider the space $Y^{\coprod 2N}$ with the action of $H_{N}^{+}$ coming from the identification
\begin{equation*}
C(Y^{\coprod 2N}) \cong C(Y, \C^{2N}) \cong C(Y)\otimes \C^{2N}
\end{equation*}
and the standard action of $H_{N}^{+}$ on the second tensor. Then, we consider the following C*-subalgebra
\begin{equation*}
A_{Z_{1}} = \left\{f\in C(Y, \C^{2N}) \mid f(x)\in \mathrm{span}\{(e_{i} + e_{i+N}), 1\leqslant i\leqslant N\}, \text{for all } x\in Z_{1}\right\}.
\end{equation*}
In other words, a function $f = \sum_{i}f_{i}\otimes e_{i}$ is in $A_{Z_{1}}$ if $f_{i} = f_{i+N}$ for all $1\leqslant i\leqslant N$. We claim that $A_{Z_{1}}$ is globally invariant under the action. Indeed, for $f\in A_{Z_{1}}$ and $x\in Z_{1}$, we have
\begin{align*}
(\mathrm{ev}_{x}\otimes \id\otimes \id)\circ\alpha(f) & = \sum_{i=1}^{N}\sum_{j=1}^{N}f_{i}(x)\otimes e_{j}\otimes \frac{v_{ji} - u_{ji}}{2} + \sum_{i=1}^{N}\sum_{j=N+1}^{2N}f_{i}(x)\otimes e_{j}\otimes \frac{v_{ji} + u_{ji}}{2} \\
& \phantom{ =} + \sum_{i=N+1}^{2N}\sum_{j=1}^{N}f_{i}(x)\otimes e_{j}\otimes \frac{v_{ji} + u_{ji}}{2} + \sum_{i=N+1}^{2N}\sum_{j=N+1}^{2N}f_{i}(x)\otimes e_{j}\otimes \frac{v_{ji} - u_{ji}}{2} \\
& = \sum_{i=1}^{N}\sum_{j=1}^{N}f_{i}(x)\otimes e_{j}\otimes \frac{v_{ji} - u_{ji}}{2} + \sum_{i=1}^{N}\sum_{j=1}^{N}f_{i}(x)\otimes e_{j+N}\otimes \frac{v_{ji} + u_{ji}}{2} \\
& \phantom{ =} + \sum_{i=1}^{N}\sum_{j=1}^{N}f_{i+N}(x)\otimes e_{j}\otimes \frac{v_{ji} + u_{ji}}{2} + \sum_{i=1}^{N}\sum_{j=1}^{N}f_{i+N}(x)\otimes e_{j+N}\otimes \frac{v_{ji} - u_{ji}}{2} \\
& = \sum_{i=1}^{N}\sum_{j=1}^{N}f_{i}(x)\otimes (e_{j}+e_{j+N})\otimes \frac{v_{ji} - u_{ji}}{2} \\
& \phantom{ =} + \sum_{i=1}^{N}\sum_{j=1}^{N}f_{i}(x)\otimes (e_{j}+e_{j+N})\otimes \frac{v_{ji} + u_{ji}}{2}
\end{align*}
from which the claim follows.
\end{ex}

One can of course also collapse orbits to the trivial one, by using the same procedure as for Huang actions: given closed subsets $Z_{0}\subset Z_{1}\subset Y$, simply set
\begin{equation*}
A_{Z_{0}, Z_{1}} = \{f\in A_{Z_{1}} \mid f(x) \in \C\xi \text{ for all } x\in Z_{0}\}.
\end{equation*}
The fact that that subalgebra still carries an action is proven in the same way as for Huang actions in \cite{huang2013faithful}. We are now ready to describe all actions of $H_{N}^{+}$.

\begin{prop}
Let $X$ be a compact Hausdorff space with a non-trivial action of $H_{N}^{+}$. Then, there exists a compact Huasdorff space $Y$, and closed subsets $Z_{0}\subset Z_{1}\subset Y$ such that the action is equivariantly isomorphic to the standard action of $H_{N}^{+}$ on $A_{Z_{0}, Z_{1}}$.
\end{prop}

\begin{proof}
Restricting the action to $S_{N}^{+}$, we already know by Theorem~\ref{thm:allactionsSn+} that $X$ is homeomorphic to $Y^{\coprod N}/Z_{0}$ for some compact Hausdorff space $Y$ and $Z_{0}\subset Y$ closed. Moreover, any orbit of size $N$ in the original action is mapped bijectively and equivariantly onto an orbit of size $N$ of the restricted action, while orbits of size $2N$ split into two orbits of size $N$. In particular, $Z_{0}$ is the set of fixed points of the original action.

Let us now consider one copy of $Y$. It contains one point of each orbit of size $N$, and two points of each orbit of size $2N$. Considering the restricted action of $H_{N} = \Z_{2}^{N}\rtimes S_{N}$, we know that there is one copy of $\Z_{2}$ which acts on $Y$ by exchanging the two points of any $2N$-orbits and fixing everything else. As a consequence, the set $Z_{1}$ of points belonging to an orbit of size $1$ or $N$ is closed and contains $Z_{0}$. We therefore have a surjection $Y^{\coprod 2N}\to X$ exactly as in the example above. The end of the proof now follows by the same arguments as for Theorem~\ref{thm:allactionsSn+} since it is enough to work orbit-wise.
\end{proof}

One important fact in the result above is that to describe all faithful actions of a compact quantum group, one needs to understand all its ergodic actions, including the non-faithful ones. In other words, faithfulness is not a good property to consider for the purpose of classifying actions, compared to ergodicity.

\subsection{Free products}

There is one last easy quantum group that we have not considered yet, namely $B_{N}^{+\sharp} = B_{N}^{+}\ast \Z_{2}$. There does not seem to be a clear way of describing actions of a free product in terms of actions of the factors in general, but here we only consider classical actions and we already know that $B_{N}^{+}$ can only act trivially on classical spaces. Moreover, any action of $\Z_{2}$ yields by induction an action of $B_{N}^{\sharp +}$ and it is therefore natural to wonder whether these are the only ones. The next result shows that this is indeed the case.

\begin{prop}
Any action of $B_{N}^{\sharp +}$ on a classical compact Hausdorff space is induced from an action of $\Z_{2}$.
\end{prop}

\begin{proof}
The proof is very close to that of Theorem~\ref{thm:ergodichn+} but things are even simpler since the representation theory of free products is more rigid than that of $H_{N}^{+}$. Recall from \cite{wang1995free} that the irreducible representations of $B_{N}^{\sharp+}$ are of the form $\w = \omega^{\epsilon_{1}}u_{n_{1}}\cdots \omega^{\epsilon_{k}}u_{n_{k}}\omega^{\epsilon_{k+1}}$ where $\omega$ is the non-trivial representation of $\Z_{2}$ (i.e. a self-adjoint unitary satisfying $\D(\omega) = \omega\otimes \omega$). 

We first assume that the action is ergodic. It was proven in \cite[Thm 6.7]{freslon2021tannaka} that given a representation $\w$, $B_{\w}\cdot B_{\w}$ is contained in
\begin{equation*}
B' = B_{\emptyset}\oplus B_{\omega}\oplus B_{u_{1}} \oplus B_{\omega u_{1}}\oplus B_{u_{1}\omega}\oplus B_{\omega u_{1}\omega}.
\end{equation*} 
Let us consider for the moment $\w = \omega u_{1}$. Because $u_{1}\otimes \w$ and $\w\otimes u_{1}$ have no common subrepresentation (the first term cannot be the same), $B_{u_{1}}\cdot B_{\w} = \{0\}$ and similarly for $B_{\omega}\cdot B_{\w}$ and $B_{u_{1}\omega}\cdot B_{\w}$. Comparing now
\begin{equation*}
\omega u_{1}\omega\otimes \w = \omega u_{2} \oplus \omega \quad \& \quad \w\otimes \omega u_{1}\omega = \omega u_{1}\omega u_{1}\omega
\end{equation*}
also yields $B_{\omega u_{1}\omega}\cdot B_{\w} = \{0\}$. The same arguments apply to $\w = u_{1}\omega$ and $\w = \omega u_{1}\omega$, eventually proving that $B'$ is a C*-algebra. As a consequence, we have an action on a finite space, which must therefore come from a $*$-homomorphism $C(S_{d}^{+})\to C(B_{N}^{\sharp +})$ for some $d\in \N$. Dual quantum subgroups correspond to sub-fusion rings, and an inspection of the fusion rules of $B_{N}^{\sharp +}$ shows that the possible dual quantum subgroups are given by the C*-subalgebras $\C$, $C^{*}(\Z_{2})$, $C(PB_{N}^{\sharp+})$ and $C(B_{N}^{\sharp+})$. The last two ones are not generated by projections because their abelianisations are not, hence the action is either trivial or induced from $\Z_{2}$. In both cases, $B'\subset B_{\emptyset}\oplus B_{\omega}$.

Observe now that for any word $\w$, any subrepresentations of $\omega\otimes \w$ and $\w\otimes \omega$ respectively do not start with the same letter, hence $B_{\omega}\cdot B_{\w} = \{0\}$. This shows that if $\w$ is symmetric in the sense that $\overline{\w} = \w$, then $B'\oplus B_{\w}$ is a C*-algebra and the previous argument yields $B_{\w} = \{0\}$. Is $\w$ is not symmetric, we consider $B'\oplus B_{\w}\oplus B_{\overline{\w}}$ and observe as in the proof of Theorem~\ref{thm:ergodichn+} that $B_{\w}\cdot B_{\overline{\w}}$ is contained in a sum of spectral subspaces corresponding to symmetric words, hence contained in $B'$ so that we can again conclude that $B_{\w} = \{0\}$. Summing up, the only non-zero spectral subspace is $B_{\omega}$, which exactly means that the action is induced from $\Z_{2}$.

If the action now is arbitrary, let us consider a function $f\in B_{\w}$ for some word $\w\neq \omega$ and $x\in X$. Restricting to the orbit of $x$, we have an ergodic action, hence $f_{\mid\mathcal{O}(x)} = 0$ by the first part of the proof. Since this holds for all $x$, we conclude that $f = 0$. In other words, the only non-trivial spectral subspace is once again $B_{\omega}$, and the proof is complete.
\end{proof}

\bibliographystyle{smfplain}
\bibliography{quantum}

\end{document}